\documentclass[]{siamonline1116}

\usepackage{lipsum}
\usepackage{amsfonts}
\usepackage{graphicx}
\usepackage{epstopdf}
\usepackage{algorithmic}
\ifpdf
  \DeclareGraphicsExtensions{.eps,.pdf,.png,.jpg}
\else
  \DeclareGraphicsExtensions{.eps}
\fi

\numberwithin{theorem}{section}

\newcommand{\TheTitle}{Weighted Total Variation Clustering: model and analysis} 
\newcommand{\TheAuthors}{G.D.Xu, Y.Xia, H.Ji}

\headers{\TheTitle}{\TheAuthors}

\title{{\TheTitle}\thanks{Submitted to the editors DATE.
\funding{This work was funded by the Fog Research Institute under contract no.~FRI-454.}}}

\author{
  Guodong Xu\thanks{Department of Mathematics, National University of Singapore, 
    (\email{a0109978@u.nus.edu}, \email{matjh@nus.edu.sg}).}
  \and
  Yu Xia\thanks{Department of Mathematics, Zhejiang University (\email{xiayu03235@126.com}).}
  \and
  Hui Ji\footnotemark[1]
}

\usepackage{amsopn}

\ifpdf
\hypersetup{
  pdftitle={\TheTitle},
  pdfauthor={\TheAuthors}
}
\fi

\usepackage{graphicx} 
\usepackage{subfig} 
\usepackage{caption}
\usepackage{enumitem}
\usepackage{empheq}
\usepackage{wrapfig}
\usepackage{multirow}
\usepackage{color}
\usepackage{tikz}
\usepackage{mathtools}

\newsiamremark{remark}{Remark}
\newcommand{\argmin}{\operatornamewithlimits{argmin}}
\newcommand{\dist}{\mathrm{dist}}
\newcommand{\CS}{\mathcal{S}}
\newcommand{\dia}{\mathrm{dia}}

\newcommand{\BR}{\mathbb{R}}
\newcommand{\sign}{\mathrm{sign}}
\newcommand{\tr}{\mathrm{tr}}
\newcommand{\etaln}{\mbox{\emph{et al.\ }}}
\newcommand{\ie}{\mbox{\emph{i.e.\ }}}
\newcommand{\eg}{\mbox{\emph{e.g.\ }}}

\definecolor{c1}{rgb}{0.968,0.461,0.4218}
\definecolor{c2}{rgb}{0.957,0.387,0.883}
\definecolor{c3}{rgb}{0,0.723,0.219}
\definecolor{c4}{rgb}{0,0.742,0.766}
\definecolor{c5}{rgb}{0.175,0.56,0.996}
\definecolor{c6}{rgb}{0.715,0.621,0}

\newcommand{\hwplotA}{\raisebox{2pt}{\tikz{\draw[c1,solid,line width=1.2pt](0,0) -- (5mm,0);}}}

\newcommand{\hwplotG}{\raisebox{2pt}{\tikz{\draw[c5,densely dashed,line width=1.2pt](0,0) -- (5mm,0);}}}
\newcommand{\hwplotH}{\raisebox{2pt}{\tikz{\draw[c3,dashed,line width=1.2pt](0,0) -- (5mm,0);}}}

\title{Weighted total variation  based convex clustering}
\author{Guodong Xu%
  \thanks{Department of Mathematics, National University of Singapore, 
    (\email{a0109978@u.nus.edu}, \email{matjh@nus.edu.sg}).}
  \and
  Yu Xia\thanks{Department of Mathematics, Zhejiang University (\email{xiayu03235@126.com}).}
  \and
  Hui Ji\footnotemark[1]}

\begin{document}

\maketitle

\begin{abstract}
Data clustering is a fundamental problem with a wide range of applications. Standard methods, \eg the $k$-means method, usually require solving a non-convex optimization problem.  Recently, total variation based convex relaxation to the  $k$-means model has emerged as an attractive alternative for data clustering. However, the existing results on its exact clustering property, \ie, the condition imposed on data so that the method can provably give correct identification of all cluster memberships, is only applicable to very specific data and   is also much more restrictive than that of some other methods. This paper aims at the revisit of total variation based convex clustering, by proposing a weighted sum-of-$\ell_1$-norm relating convex model. Its exact clustering property established in this paper, in both deterministic and probabilistic context, is applicable to general data and is much sharper than the existing results. These results provided good insights to advance the research on convex clustering. Moreover, the experiments also demonstrated that the proposed convex model has better empirical performance when be compared to standard clustering methods, and thus it can see its potential in practice.
\end{abstract}


\section{Introduction} 
\label{sec:introduction}
The task of data clustering is to group data into a set of groups (called clusters), such that data objects in the same cluster are similar to each other, while data objects in different clusters are dissimilar from one another. More specifically, considering $m$ feature vectors $\{A_1,A_2,\ldots, A_m\}\subset \BR^n$. We assemble a data matrix,
denoted by $A\in\mathbb{R}^{m\times n}$, whose rows represent feature vectors. The task of clustering is then to partition these $m$ row vectors into different groups based on their dissimilarity (\ie distance) in some metric space~$(M,d)$, where $M$ is the set the observations live in and $d: M\times M\rightarrow\mathbb{R}$ is the function that quantizes the dissimilarity. The most often seen distances include $\ell_2$-norm~(\emph{$k$-means}~model~\cite{macqueen1967some}) and~$\ell_1$-norm (\emph{$k$-median} model~\cite{jain1988algorithms}).
In other words, a clustering method is to partition the set of the rows of $A$ into $k$ groups: 
\[
\{\mathcal{S}^{(0)},\mathcal{S}^{(1)},\cdots,\mathcal{S}^{(k-1)}\}
\]
where $\mathcal{S}^{(i)}=\{A_{j_k}\}_{s=0}^{m_i-1}$, such that the distances among the rows in the same group are small and the distances between any two rows from two different groups are large.

Data clustering is a fundamental problem with a wide range of applications, including data mining,  machine learning, pattern recognition, image analysis, information retrieval, bioinformatics, data compression, and computer graphics. See \eg \cite{coleman1979image,kaufman2009finding,gong2013fuzzy,filipovych2011semi,huth2008classifications,bewley2011real,basak1988determining}).  
In the past, many clustering algorithms have been developed, such as \emph{$k$-means clustering}, \emph{hierarchical clustering}, \emph{spectral clustering} and many others (see \cite{xu2005survey} for a survey). 
Among them, $k$-means clustering arguably is the prominent one used in practice. The basic idea of $k$-means method  \cite{macqueen1967some} is to partition the $m$ feature vectors $\{A_i\}_{i=1}^m$ into $k$ clusters with a minimal sum of inner-class distances. It can be formulated as the following  optimization problem:
\begin{equation}\label{eqn:kmeans}
\begin{split}
\min_{\{\CS^{(i)}\}_{i=0}^{k-1}}\sum_{i=0}^{k-1}\sum_{A_{j}\in\CS^{(i)}}\|A_j-c_i\|_2^2\quad\mbox{s.t.} \quad c_i = \frac{1}{|\CS^{(i)}|}\sum_{A_{j}\in\CS^{(i)}}A_j. 
\end{split}
\end{equation}
The optimization problem above is NP-hard, and it is shown in~\cite{lindsten2011just} that this problem is equivalent to  a non-convex optimization problem with $\ell_0$-norm relating constraints:
\begin{equation}\label{eqn:nonconvex_equivalent}
\widehat{X}=\mathrm{argmin}_{X\in\BR^{m\times n}}\|A-X\|_F^2, \quad
\mbox{s.t.}\quad k=m-\|\alpha\|_0,
\end{equation}
where 
\[
\left\{
\begin{array}{lcl} 
\alpha&=&[\alpha_2,\cdots,\alpha_m],\\
\alpha_j&=&j-1-\|\beta^{(j)}\|_0,\\
\beta^{(j)}&=&[\beta_1^{(j)},\cdots,\beta_{j-1}^{(j)}],\\
\beta_i^{(j)}&=&\|X_i-X_j\|_2.
\end{array}
\right.
\]
The minimizer  $\widehat{X}$  contains $k$ unique unique rows, say $\{\widehat{X}_{j_k}\}_{s=0}^{k-1}$,  then the $m$ observations are grouped into $k$ clusters based on their distances to $\{\widehat{X}_{j_k}\}_{s=0}^{k-1}$: 
\[
\CS^{(s)}=\{A_{j}\ |\ \|A_{j}-\widehat{X}_{j_k}\|_2\leq\|A_{j}-\widehat{X}_{j_{s'}}\|_2,\ s'\neq s\},\quad s=0,1,\cdots,k-1.
\]

\subsection{TV-based convex relaxation}
Clearly, the optimization problem~\eqref{eqn:nonconvex_equivalent} is non-convex, and thus it is difficult to find a global minimizer. Moreover, the outcome of its solver usually is sensitive to the initialization.  
 Similar to the $\ell_1$-norm relating convex relaxation used in compress sensing (see \eg\cite{candes2006robust}) and the total variation for image recovery (see \eg\cite{Rudin1992}),
 the convex relaxation of the  problem \eqref{eqn:nonconvex_equivalent} can be done by replacing $\ell_0$-norm by convex norms. Lindsten~\etaln\cite{lindsten2011just} proposed the following convex sum-of-$\ell_p$-norm  model,
\begin{equation}\label{eqn:convex}
\min_{X\in \BR^{m\times n}}\|A-X\|_F^2+c\sum_{j=2}^{m}\sum_{i<j}\|X_i-X_j\|_p,
\end{equation}
where $p\geq 1$.
In~\eqref{eqn:convex}, the sum--of-$\ell_p$-norm term refers to the total variation (TV) of $X$ along columns,
$c$ is a tuning parameter which determines the number of unique rows of the minimizer (equivalent to the number of output clusters). The $\ell_p$-norm relating optimization problem~\eqref{eqn:convex} has been  extensively studied in recent years, and  there exist many efficient numerical solvers, such as ADMM method \cite{chi2014convex} and AMA method~\cite{chi2014splitting}.  Regarding the choice of the parameter $c$, Hocking~\etaln\cite{hocking2011clusterpath} proposed a selection strategy of $c$. It was shown in \cite{hocking2011clusterpath} that  the convex model~\eqref{eqn:convex} outperforms several existing clustering techniques, including $k$-means and single linkage clustering, in certain applications.

Theoretical analysis on the model \eqref{eqn:convex} with $p=2$,
\begin{equation}\label{eqn:convex2}
\min_{X\in \BR^{m\times n}}\|A-X\|_F^2+c\sum_{j=2}^{m}\sum_{i<j}\|X_i-X_j\|_2,
\end{equation}
 is first presented in Zhu~\etaln\cite{zhu2014convex}, in which   an exact clustering condition is established for 
  the data with only two clusters ($k=2$). For a pair of $n$-dimensional clusters, denoted by $\CS^{(i)}$ and $\CS^{(j)}$, define their
\emph{intra-class distance} by
\begin{equation}\label{eqn:intradist}
\dist(\CS^{(i)},\CS^{(j)})={\inf}_{x\in \CS^{(i)}, y\in \CS^{(j)}}\|x-y\|_2.
\end{equation}
For a   cluster $\mathcal{S}$, define its
\emph{inner-class distance} (diameter of cluster $\mathcal{S}$) by
\begin{equation}\label{eqn:innerdist}
\dia(\CS)=\sup_{x\neq y\in \CS} \|x-y\|_2.
\end{equation}
Zhu~\etaln\cite{zhu2014convex} showed  that  two clusters, $\CS^{(0)}$ and $\CS^{(1)}$, can be exactly clustered via \eqref{eqn:convex}, provided that
\begin{equation}\label{eqn:zhu}
\dist(\CS^{(0)},\CS^{(1)})>
\max_{s=0,1}\left(1+2\frac{m_{1-s}(m_k-1)}{m_k^2}\right)\dia (\CS^{(s)}),
\end{equation}
where $m_k$ denotes the sample size of the cluster $\CS^{(s)}$. 

\subsection{Linear and semi-definite programming based convex relaxation}
There are also other types of convex relaxations for  $k$-means model. Awasthi \etaln\cite{awasthi2015relax} proposed both linear programming~(LP) based relaxation and semidefinite programming (SDP) based relaxation for $k$-means model. They rewrote the $k$-means problem as the following optimization problem:
\begin{equation}
\label{eqn:mode:kmeansequilvalent}
\min_{Z\in\BR^{m\times m}} \tr(Z^THZ) \quad \mathrm{s.t.}\quad  \tr(Z)=k,\ \vec{\textbf{1}}^TZ=\vec{\textbf{1}}^T,\ Z_{i,i}>0,\ \mathrm{and}\ Z_{i,j}\in\{0,Z_{i,i}\},
\end{equation} 
where $H$ is a $m\times m$ square matrix whose $(i,j)$-th entry is $\|A_i-A_j\|_2^2$, $\vec{\textbf{1}}$ is the all-ones vector. 
The object function of~\eqref{eqn:mode:kmeansequilvalent} is convex, but the feasible set is  non-convex due to the discrete constrains $Z_{i,j}\in\{0,Z_{i,i}\}$. By relaxing the discrete constrains $Z_{i,j}\in\{0,Z_{i,i}\}$ to the interval constrains $Z_{i,j}\in[0,1]$, it leads to LP relaxation model
\begin{equation}
\label{eqn:model:LP}
\min_{Z\in\BR^{m\times m}} \tr(Z^THZ) \quad \mathrm{s.t.}\quad  \tr(Z)=k,\ \vec{\textbf{1}}^TZ=\vec{\textbf{1}}^T,\ \mathrm{and}\  Z_{i,j}\in[0,1].
\end{equation} 
By relaxing the discrete constrains to  $Z\succeq 0$, it
leads to  the following SDP relaxation model:
\begin{equation}
\label{eqn:model:SDP}
\min_{Z\in\BR^{m\times m}} \tr(Z^THZ) \quad \mathrm{s.t.}\quad  \tr(Z)=k,\ \vec{\textbf{1}}^TZ=\vec{\textbf{1}}^T,\ \mathrm{and}\  Z\succeq 0.
\end{equation} 

For LP and SDP models, the exact clustering conditions are established in probabilistic context, \ie, the exact clustering with a overwhelming probability. 
In \cite{awasthi2015relax,iguchi2015tightness}, the input data matrix $A$ is sampled from the stochastic ball model, see \cref{def:stochastic_ball}.
Denote 
\begin{equation}
\label{eqn:center_distance}
\Delta=\min_{i\neq j} \|\mu_i-\mu_j\|_2
\end{equation}
as the smallest distance between the ball centres. It is shown in \cite{awasthi2015relax}  that the $k$-means LP model can have exact clustering with high probability only in the regime $\Delta\geq 4$. The $k$-means SDP model can have exact clustering  with high probability  in the regime 
$\Delta>2\sqrt{2}(1+\sqrt{1/n})$. A tighter condition is presented in \cite{iguchi2015tightness}, which says that the $k$-means SDP relaxation model can recover the planted clusters in stochastic ball model with high probability, provided with 
\begin{equation}
\label{eqn:SDP_sep}
\Delta >2+\frac{k^2}{n}\mbox{Cond}(\mu),\end{equation} where  
$\mbox{Cond}(\mu)=\max_{i\neq j}\|\mu_i-\mu_j\|_2^2\large/\min_{i\neq j}\|\mu_i-\mu_j\|_2^2$.

\subsection{Motivations and our contributions}
\label{sec:motivations}
In the configuration of only $2$ clusters, 
Zhu \etaln \cite{zhu2014convex} show that the TV relating model \eqref{eqn:convex} can  exactly cluster the input data provided the data satisfies the condition ~\eqref{eqn:zhu}. The condition ~\eqref{eqn:zhu} is much more restrictive than that of most existing approaches. Also it is 
dependent on cluster sizes, which makes it not not suitable for 
the data set with large variance on cluster sizes.  It leads to the impression 
that TV-based  convex relaxation is less likely to succeed than other methods. Such an impression contradicts to the good performance of TV-based convex clustering in practice.

In this paper, the TV relating model for convex clustering is revisited. The aim is to  show  that the theoretical soundness of TV relating convex clustering is not necessarily more restrictive than others, and has its potential in practice.
Motivated by the
performance gain when using a weighted version of the sum-of-$\ell_2$-norm based model  \eqref{eqn:convex2} with Gaussian kernel based scheme \cite{lindsten2011clustering,lindsten2011just,hocking2011clusterpath}, 
we proposed to study the following weighted   total variation based model with the sum-of-$\ell_1$-norm for clustering:
\begin{equation}
\label{eqn:weightedconvex}
\widehat{X}=\mathrm{argmin}_X\|A-X\|^2_F+c\sum_{j=2}^m\sum_{i<j}e^{-r\|A_i-A_{j}\|_2^2}\|X_i-X_j\|_1
\end{equation}
where 
$\sum_{j=2}^m\sum_{i<j}e^{-r\|A_i-A_{j}\|_2^2}\|X_i-X_j\|_1$ is the   weighted total variation of $X$ along columns.

For the optimization model 
\eqref{eqn:weightedconvex}, we first showed that for the data with 
only 2 clusters, its minimizer $\widehat{X}$ can exactly recover the two clusters, as long as the intra-class distance is larger than  inner-class distances and is independent to cluster size. Such an exact clustering condition is much less restrictive than the condition presented in \cite{zhu2014convex}. Furthermore, we also established the exact clustering condition of the model 
\eqref{eqn:weightedconvex} for the data with arbitrary number of clusters. With some mild condition on cluster centers, the same condition is established for the data with any number of clusters.
At last, the experiments are done  several applications which showed the proposed convex clustering model \eqref{eqn:weightedconvex} outperformed several existing well-known techniques, particularly on the datasets with large variance on cluster sizes.

In short,
The results presented in this paper presented a deep theoretical understanding of  total variation  relating convex clustering, but also provided a practical clustering technique which can see  its applications in many problems.

\subsection{Notations and definitions}
\label{sec:notation}
before proceeding the discussion, we present some  notations and definitions which are used in the remainder of the paper. We use capital letters for matrices, \eg matrix $A$, calligraphic capital letters for sets, \eg index set $\mathcal{I}$, bold small letters with a right arrow overset for column vectors, \eg $\vec{\textbf{a}}\in \mathbb{R}^m$, and small letters for constant, \eg $r\in \mathbb{R}^+$. Moreover, $A_p$ denotes the $p$-th row of $A$, $A_{pq}$ denotes the entry at $p$-th row and $q$-th column, and $\vec{\bf{a}}_i$ denotes the $i$-th element of $\vec{\bf{a}}$. $\|\cdot\|_{F}$ and $\|\cdot\|$ are the Frobenious norm and
the operator norm respectively, $\mbox{Tr}(\cdot)$ is the trace of
the matrix. For a vector $\vec{\textbf{a}}$, the $\ell_1,\ell_2$ and $\ell_p$ norms are denoted by $\|\vec{\textbf{a}}\|_1$, $\|\vec{\textbf{a}}\|_2$ and $\|\vec{\textbf{a}}\|_p$, respectively.
	
Given a data matrix $A\in\mathbb{R}^{m\times n}$ sampled from $k$ clusters. Assume there are totally $m_i$ rows of $A$ sampled from the set $\mathcal{S}^{(i)}$, with $0\leq i\leq k-1$. Without loss of generality, let $\mathcal{I}^{(s+1)}=\{m_{s}+1,m_k+2,\dots,m_k+m_{s+1}\}$ denote the indices of sampling in $\mathcal{S}^{(s)}$, for $s=1,2,\dots,k-1$ respectively. And $\mathcal{I}^{(0)}=\{1,2,\dots,m_0\}$.

Now we introduce the total difference operator for notation simplicity. 
\begin{definition}[total difference operator]
	\label{def:total}
	Let $G^i=(\textbf{0}_{(i-1)\times(m-i)},\vec{\textbf{1}}_{(i-1)\times 1}, -I_{(i-1)\times(i-1)})$. The total difference operator of order $m$ is defined by
	$$D^m=[(G^{m})^T, (G^{m-1})^T,\dots, (G^2)^T]^T.$$
	Let \[\mathcal{D}=\{D^mX\ |\ X\in \mathbb{R}^{m\times n}\},\]
	$D^m$ is a linear operator from $\mathbb{R}^{m\times n}$ to $\mathcal{D}$. Moreover, for any $Y\in R^{\binom{m}{2}\times m}$, $Y$ belongs to $\mathcal{D}$ if and only if $HY=0$, where $H=(D^{m-1},I)$.
\end{definition}

For simplicity,  we take $D=D^m$ unless otherwise stated. Then, by the definition of $\mathcal{I}^{(s)}$ and $D$, we define the index set $\mathcal{P}^{(0)}$ representing  the difference of two rows from the same cluster:
	\begin{equation} \label{eqn:P0}
	\mathcal{P}^{(0)}=\{(i-1)(m-1)+(j-i)\ |\ i,j\in \mathcal{I}^{(s)},\ 0\leq s\leq k-1\},
	\end{equation}
	and the index set 
	\begin{equation}\label{eqn:P1}
	\mathcal{P}^{(1)}=\{(i-1)(m-1)+(j-i)\ \ |\ \ i\in \mathcal{I}^{(s)},\ j\in \mathcal{I}^{(l)},\ s<l\}
	\end{equation} representing the row indexes of $DX$ which correspond to the difference of two rows from different clusters.
		Besides, $\sign(x)$ denotes the signum function of a real number $x$:
	\[
	\sign(x)=
	\left\{\begin{array}{ll}
	-1 & \mbox{if} \ x<0,\\
	0 & \mbox{if} \ x=0,\\
	1 & \mbox{if} \ x>0.
	\end{array}
	\right.
	\]

\section{Main results}
\label{sec:result}
Consider $k$ clusters in $\mathbb{R}^{n}$, denoted by $\mathcal{S}^{(s)}, s=0,1,\ldots,k-1$. Let $A\in \mathbb{R}^{m\times n}$ denote the data matrix whose rows are arbitrarily chosen from these $k$ clusters.
\begin{definition}[exact clustering property]
	A matrix $X\in \mathbb{R}^{m\times n}$ is said to have {exact clustering property}, if it satisifies
	\begin{equation}\label{eqn:exact}
	\left\{\begin{array}{ll}
	X_i=X_j, & \mbox{if}\  A_i, A_j\in \CS^{(s)},\\
	X_i\neq X_j, & \mbox{otherwise,}
	\end{array}
	\right.
	\end{equation}
	for any  $1\leq i\neq j\leq m$ and $s=0,1,\dots,k-1$.
\end{definition}

Clearly, given  a  matrix $X$ with exact clustering property, one can correctly group the rows of $X$ into $k$ clusters, by putting the rows with same value into one group, \ie, given a data matrix $A$, a convex model will exactly recover the clusters from $A$ if its minimizer has exact clustering property. Recall that the model \cref{eqn:weightedconvex} reads
\begin{equation}\label{eqn:weightconvex}
\hat{X}=\operatornamewithlimits{argmin}_X\|A-X\|^2_F\\
+c\sum_{i<j}e^{-r\|A_i-A_{j}\|_2^2}\|X_i-X_j\|_1.
\end{equation}

\subsection{Clustering data with $k$ clusters}
Before presenting the main results for the data with $k(\geq 2)$ clusters, we first discuss it in the  $2$-cluster setting for illustrating the main idea.

\begin{theorem}[Exact clustering for $2$ clusters]
	\label{theorem_twobody}
	Let $A\in \mathbb{R}^{m\times n}$ be a data matrix whose rows are arbitrarily sampled from two sets $\mathcal{S}^{(0)}, \mathcal{S}^{(1)}\subset \mathbb{R}^{n}$.
	Suppose that $\mathcal{S}^{(0)}$ and $\mathcal{S}^{(1)}$ satisfy the following separation condition:
	\begin{equation}\label{eqn:sep1}
	\dist(\mathcal{S}^{(0)},\mathcal{S}^{(1)})>\max\{\dia(\mathcal{S}^{(0)}), \dia(\mathcal{S}^{(1)})\}.
	\end{equation}
	Then,  there exist some constants $c,r>0$ such that the minimizer $\widehat X$  of \cref{eqn:weightconvex}
	has exact clustering property \cref{eqn:exact}, \ie for any $1\leq i\neq j\leq m$,
	$$
	\left\{\begin{array}{ll}
	\widehat{X}_i=\widehat{X}_j, & \mbox{if } \{A_i, A_j\}\subset \CS^{(s)} \mbox{ for } 0\leq s\leq 1,\\
	\widehat{X}_i\neq \widehat{X}_j, & \mbox{otherwise.}
	\end{array}
	\right.
	$$
\end{theorem}
\begin{remark}
Clearly, in the $2$-cluster setting,
our exact clustering condition \eqref{eqn:sep1} is much less restrictive than
the condition \cref{eqn:zhu} for model \eqref{eqn:convex}.
\end{remark}
\begin{remark}[Range of feasible parameters for $2$-cluster]
As shown in the proofs of  \cref{theorem_twobody}, the model  \cref{eqn:weightconvex} guarantees exact clustering, if two parameters $c, r$ are chosen such that
	$$
	r\geq \max \limits_{0\leq i\leq 1}{\ln\left({4\frac{m-m_i}{m_i}}\right)}{(d^2-l_i^2)^{-1}},
	$$
	where 
	\[d=\dist(\mathcal{S}^{(0)},\mathcal{S}^{(1)})\quad \mbox{and}\quad  l_i=\dia{(\mathcal{S}^{(i)})},\]
	and 
	\[
	\underline{\kappa}(r)\leq c\leq \overline{\kappa}(r).
	\]
	The upper bound $\overline{\kappa}(r)$ of $c$ is given by
		\[
		\min_{p\in\mathcal{P}^{(1)}}\left\{\min_{q\in\mathcal{Q}}\frac{2|\tau_q|}{|{m}(\rho-\gamma_p)|},\min_{q\in\mathcal{Q}}\frac{2|\tau_q|}{m\rho}\right\},
		\]
		and the lower bound $\underline{\kappa}(r)$ of $c$ is given by
		\[
		\max_{0\leq i\leq 1,p\in{\mathcal{P}}^{(0)}}\left\{\frac{\epsilon_i\dia(\mathcal{S}^{(i)})}{\gamma_p-\frac{4(m-m_i)}{m_i}\max\limits_{i\in\mathcal{P}^{(1)}}(\gamma_i)}\right\},
		\]
where  $\tau_q=\frac{1}{m_0m_1}\sum_{i\in\mathcal{P}^{(1)}}(DA)_{iq}$,
		$\epsilon_i=\frac{8(m-m_i)(m_i-1)+4m_i^2}{mm_i^2}$, $\mathcal{Q}=\{q|\tau_q\neq 0\}$ and $\rho=\frac{\sum_{i\in\mathcal{P}^{(1)}}\gamma_i}{m_0m_1}$.

From the explicit form of $\overline\kappa(r)$ and $\underline\kappa(r)$, the length of feasible interval $\Delta \kappa(r)=\overline{\kappa}(r)-\underline{\kappa}(r)$ has the property that $\lim_{r\rightarrow +\infty}\Delta \kappa(r)=+\infty$. See more details in \cref{sec:prf of two-cluster}. Thus, the range of proper parameters is wide, as long as $r$ is chosen sufficiently large.
\end{remark}

The results in the $2$-cluster setting cannot be applied to  
the data with $k(\geq$3) clusters. The main difference is that when there are 3 or more clusters, there is a hidden constraint among different clusters. Taking $3$ clusters in one-dimensional as an example, denote $\hat{x}_{i}$ as the mean value of cluster $\mathcal{S}^{(i)}$, for $i=0,1,2$. Considering the mean value difference between clusters as $\hat{y}^{(i,j)}=\hat{x}_{i}-\hat{x}_{j}$, they satisfy 
\begin{equation}\label{eqn:hidconstraint}
\hat{y}^{(0,1)}+\hat{y}^{(1,2)}=\hat{y}^{(0,2)}.
\end{equation}
 See \cref{fig:1D3B} for  illustration.
\begin{figure}
	\centering
	\includegraphics[height=0.16\textwidth]{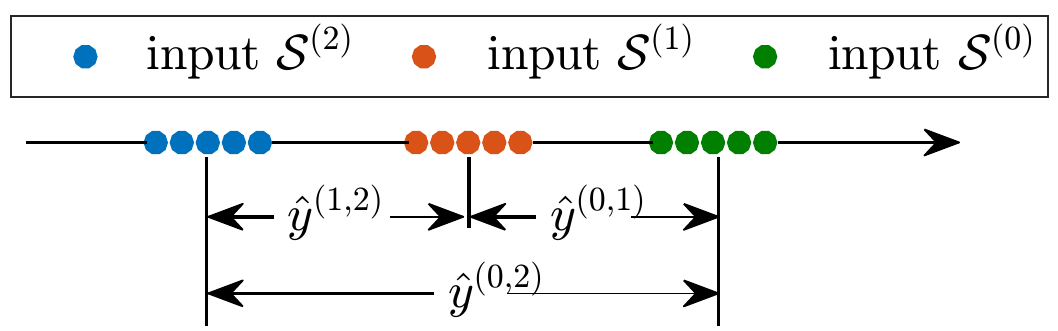}\\
	\caption{Example: three one-dimensional clusters $\{\mathcal{S}^{(i)}\}_{i=0,1,2}$.}
	\label{fig:1D3B}
\end{figure}

 With additional mild conditions on cluster centers,  we can have the similar result for the data with $k(\geq3)$
clusters. Let $m_i$ and $A_c^{(i)}$ denote the size and the mean  value of cluster $\mathcal{S}^{(i)}$, respectively, \ie $m_i=\sum_{x\in \mathcal{S}^{(i)}}1$ and $A_c^{(i)}=\frac{1}{m_i}\sum_{x\in \mathcal{S}^{(i)}}x$. Then, we have

\begin{theorem}[Exact clustering for $k$ clusters]
	\label{theorem1} 
	Let $A\in \mathbb{R}^{m\times n}$ be a data matrix whose rows are arbitrarily sampled from $k$ sets $\{\mathcal{S}^{(i)}\}_{i=0}^{k-1}\subset \mathbb{R}^{n}$. Suppose that  the cluster means $\{A^{(i)}_c\}_{i=0}^{k-1}$ are distinct in all dimensions, \ie $(A^{(i)}_c)_p\neq (A^{(j)}_c)_p$ for all $1\leq p\leq n$ and $0\leq i<j\leq k-1$. Suppose that $\{\mathcal{S}^{(i)}\}_{i=0}^{k-1}$ satisfy the following separation condition:
	\begin{equation}\label{eqn:sep}
	\min_{i\neq j}\dist(\mathcal{S}^{(i)},\mathcal{S}^{(j)})>\max_{0\leq s\leq k-1}\dia (\mathcal{S}^{(s)}).
	\end{equation}
	Then,  there exist  $c,r>0$ such that the minimizer $\widehat X$  of \cref{eqn:weightconvex} has exact clustering property.\end{theorem}

\begin{remark}[Range of feasible parameters]\label{remark1} As shown in the proofs of  \cref{theorem1}, the model  \cref{eqn:weightconvex} guarantees exact clustering, if two parameters $c, r$ are chosen such that
	$$
	r\geq \max \limits_{0\leq i\leq k-1}{\ln\left({4\frac{m-m_i}{m_i}}\right)}{(d^2-l_i^2)^{-1}},
	$$
	where 
	\[d=\max_{i\neq j}\dist(\mathcal{S}^{(i)},\mathcal{S}^{(j)})\quad \mbox{and}\quad  l_i=\dia{(\mathcal{S}^{(i)})},\]
	and 
	\[
	\underline{\kappa}'(r)\leq c\leq \overline{\kappa}'(r).
	\]
	The upper and lower bound of $c$ is similar as $2$-cluster case, where 
		\[
		\overline{\kappa}'(r)=\min_{s<l,1\leq q\leq n}\left\{\frac{|\tau^{s,l}_q|}{3m\max\limits_{j\in\mathcal{P}^{(1)}}(\gamma_j)}\right\}
		\]
		and
		\[
		\underline{\kappa}'(r)=\max_{0\leq i\leq k-1,p\in \mathcal{P}^{(0)}}\left\{\frac{\epsilon_i\dia(\mathcal{S}^{(i)})}{\gamma_p-\frac{4(m-m_i)}{m_i}\max\limits_{j\in \mathcal{P}^{(1)}}(\gamma_j)}\right\}
		\]
				
		It is also easy to show that the feasible interval of $c$, $[\underline{\kappa}'(r),\overline{\kappa}'(r)]$ has the property that $\lim_{r\rightarrow +\infty}\overline{\kappa}'(r)-\underline{\kappa}'(r)=+\infty$; see more details in \cref{sec:prf of three-cluster}. 
\end{remark}

\subsection{Clustering data drawn from stochastic distributions}

The results in \cref{theorem1} can also be applied to the data drawn from specific stochastic distributions. For example, consider the data drawn from stochastic ball model as below.
\begin{definition}[$(\mathcal{B},\mu, m)$-stochastic ball model]
\label{def:stochastic_ball}
Let $\{\mu_i\}_{i=1}^s$ be ball centers in $\mathbb{R}^n$. For each $i$, draw i.i.d. vectors $\{r_{i,j}\}_{j=1}^{m/s}$ from some rotation-invariant distribution $\mathcal{B}$ supported on the unit ball. The points from cluster $i$ are then taken to be $x_{i,j}=r_{i,j}+\mu_{i}$.
\end{definition}

Then, we have
\begin{theorem}\label{thm:ball_model}
Let $A\in \mathbb{R}^{m\times n}$ be a data matrix drawn from stochastic ball model. Suppose that the ball centers $\mu_i(0\leq i\leq k-1)$ satisfy the separation condition $\min_{i\neq j}\|\mu_i-\mu_j\|\geq 4$, there exist  $c,r>0$ such that the minimizer $\widehat X$  of \cref{eqn:weightconvex} has exact clustering property \cref{eqn:exact}.\end{theorem}
\begin{proof}
Since the marginal distribution of $r_{i,j}$ in each dimension is continuous with non-zero variance, the sample means inside each cluster  are distinct with probability $1$ and the sufficient condition on exact recovery of model \cref{eqn:weightconvex} becomes $\min_{i\neq j}\|\mu_i-\mu_j\|\geq 4$.
\end{proof}

Different from the exact clustering condition of $k$-means LP relaxation in \cite{awasthi2015relax}, the exact clustering condition of the TV relating model \eqref{eqn:weightconvex} does not impose any assumption on  sample size and distribution shape.
Such a condition allows it can be extended to the data drawn from other stochastic distributions. For example, data drawn from mixture of Gaussian. 

\begin{definition}[$(\omega, \mu,\Sigma)$-mixture of Gaussian model]
A mixture of $k$ Gaussian distributions is a distribution in which each sample is drawn from $\mathcal{N}(\mu_{i},\Sigma_{i})$
with probability $\omega_{i}$, where $\sum_i \omega_{i}=1$. 
\end{definition}
Then, we have 
\begin{theorem}
\label{thm:GaussianMixture}
Let $A\in \mathbb{R}^{m\times n}$ be a data matrix whose rows are arbitrarily sampled from mixture of $(\omega, \mu,\Sigma)$-mixture of Gaussian model with $\mu_i(0\leq i\leq k-1)$ distinct in each dimension. Suppose the separation condition is 
\begin{eqnarray*}
\min_{i\neq j}\|\mu_{i}-\mu_{j}\|_{2} & > & \max_{i\neq j}\left(\|\Sigma_{s,l}\|^{\frac{1}{2}}\sqrt{12\log m}+{(12\log m)}^{1/4}\|\Sigma_{s,l}\|_{F}^{1/2}\right)\\
 &  & +{\max}_{i}\sqrt{\mbox{\mbox{Tr}}(\Sigma_{i})+\|\Sigma_{i}\|_{F}\sqrt{12\log m}+6\|\Sigma_{i}\|\log m},
\end{eqnarray*}
where $\Sigma_{i,j}=\Sigma_{i}+\Sigma_{j}$.
Then there exist some constants $c,r>0$ such that the minimizer $\widehat{X}$
of \cref{eqn:weightconvex} has exact clustering property
with probability larger than $1-3/m$.
\end{theorem}
\begin{remark}
Denote $\sigma_{\max}=\max_i\|\Sigma_i\|^{1/2}$. The separation condition on exact recovery of mixture of $k$ spherical Gaussian distribution becomes $\Omega(\sigma_{\max}\sqrt{n})$. 

In order to eliminate the effect of dimension $n$, partition the sample matrix $A$ uniformly into two submatrices $A_1\in \mathbb{R}^{m/2\times n}$ and $A_2\in \mathbb{R}^{m/2\times n}$. Apply weighted convex model \cref{eqn:weightconvex} to $P_{A_1}(A_2)$ and $P_{A_2}(A_1)$, where $P_X$ is a rank-$k$ projection of matrix $X$, \ie for any other matrix $Y$ of rank at most $k$, we have $\|X-P_XX\|\leq \|X-Y\|$. By the similar arguments as in \cite{AW05}, the clusters can be recovered with high probability provided with the separation condition $\Omega(\sigma_{\max}\frac{1}{\sqrt{\omega_{\min}}}+\sqrt{k\log m})$ and $\omega_{\min}=\min_i \omega_k$. The condition is not more restrictive than the existing separation condition of Gaussian mixture model; see \cite{AR2010,AS12}.
\end{remark}

\subsection{Connections to other related works on convex clustering}
\label{sec:comparison}
TV-based convex model with sum-of-$\ell_2$-norm \eqref{eqn:convex2} is studied in \cite{zhu2014convex}. However,
the exact clustering condition established in \cite{zhu2014convex}
is very restrictive and limited. It says that when the data contains only two clusters, it can be exactly clustered by the convex model \eqref{eqn:convex2} provided that the condition~\eqref{eqn:zhu} is satisfied. It can be seen that the condition~\eqref{eqn:zhu} requires the intra-class distance is at least three times of the inner-class distance, and it increases when the size ratio of two classes increases. In contrast,
the paper proposed another type TV-based convex model with weighted sum-of-$\ell_1$-norm \eqref{eqn:weightconvex}. It is shown in this paper that the exact clustering condition 
\eqref{eqn:sep1} for the
model \eqref{eqn:weightconvex}
is much weaker than the condition~\eqref{eqn:zhu}. Not only the ratio between intra-class distance and inner-class distance is reduced from $3$ to $1$, but also it is independent to cluster sizes. Furthermore, we also established the exact clustering condition for the data with $k\geq 3$ clusters, which requires all inner-class distances is smaller than any intra-class distance and cluster centers are distinct along each dimension. 

Moreover, the exact clustering condition \eqref{eqn:sep} is also less restricted than that for many existing clustering methods. Particularly, it is independent of cluster size and the number of clusters, which makes it very attractive in terms of scalability 
and un-balanced data sets. In the next, we give a more detailed discussion on exact clustering conditions of other convex relaxation of $k$-means method.

The exact clustering condition for LP-based  and SLP-based convex relaxations  are established in probability context, \ie,  exact clustering with an overwhelming probability. 
In \cite{awasthi2015relax,iguchi2015tightness}, the input data matrix $A$ is sampled from the stochastic ball model, see \ref{def:stochastic_ball}. Using notation~\eqref{eqn:center_distance}, our exact clustering condition in \cref{thm:ball_model} can be rewritten as 
\begin{equation}
\label{eqn:sep_unitball}
\Delta\geq 4.
\end{equation}
In contrast, It is shown in \cite{awasthi2015relax}  that the $k$-means LP can do exact clustering with high probability in the regime $\Delta\geq 4$, only  when all clusters of of same size. 
Such a condition makes its applicability much more limited than ours, as ours is applicable to the clusters with different sises. 
Another exact clustering condition in probability context is also given in \cite{awasthi2015relax} for the regime 
$\Delta>2\sqrt{2}(1+\sqrt{1/n}).$ 
This condition is dependent on feature dimension and is weaker than condition $\cref{eqn:sep_unitball}$ when the feature dimension is small.

It is shown in \cite{iguchi2015tightness} that the $k$-means SDP relaxation \eqref{eqn:model:SDP} can exactly recovers the planted clusters in stochastic ball model with high probability, provided the  condition~\eqref{eqn:SDP_sep} is satisfied.
This  conditions is slightly weaker than the condition \eqref{eqn:sep_unitball} for certain cases. However, when the maximal distance between centers is large and feature dimension is small, \ie $\text{Cond}(\mu)$ is large while $n$ is small, the condition in \cite{iguchi2015tightness} will be weaker than condition \eqref{eqn:sep_unitball}. In addition, the exact clustering condition \cref{eqn:SDP_sep} is dependent of the number of clusters \cref{eqn:SDP_sep}, while our condition \eqref{eqn:sep_unitball} is independent of the cluster number $k$,
which makes it more scalable. 

In short,  for the LP and SDP based convex cluster models, the  condition for exact clustering with high probability established 
is applicable to the data set with high feature dimension and sufficiently large sample sizes. It is also dependent of the number of clusters and cluster sizes.  In contrast,
 the condition \cref{eqn:sep} for the proposed weighted TV convex model  does not impose any statistical properties of the data, and  is independent of cluster sizes and number of clusters. Therefore, in terms of theoretical soundness, the weighted TV convex model is more  appealing to the data set with unbalanced cluster size or with many clusters. At last, the proposed weighted TV convex model can be applied to the data drawn from 
  mixture of Gaussians, which is not applicable to LP and SDP based approaches.
%

\subsection{Numerical solver and clustering procedure}
The problem \cref{eqn:weightconvex} is a strictly convex problem with at most one local minimiser which is also global minimizer. There are many  efficient numerical solvers for solve \cref{eqn:weightconvex}, \eg the ADMM method \cite{chi2014convex} and the AMA method~\cite{chi2014splitting}.
In this paper, we use the ADMM method to solve \cref{eqn:weightconvex}. Define the augmented lagrangian function as
\[
L_\nu(X,Z,\Lambda)=\frac{1}{2}\|A-X\|_F^2+c\sum_{l\in\varepsilon}\omega_{l}\|Z_{l}\|_1+\\
\sum_{l\in \varepsilon}\langle\Lambda_l,Z_l-X_{l_1}+X_{l_2}\rangle+\frac{\nu}{2}\sum_{l\in\varepsilon}\|Z_{l}-X_{l_1}+X_{l_2}\|^2_2,
\]
where $l=(l_1,l_2)$ with $l_1<l_2$, and $\varepsilon=\{l=(l_1,l_2):\omega_l>0\}$ is the index set corresponding to non-zero weights\footnote{For efficiency, the implementation only uses weights on the distances of each points and  its $s(=5)$ nearest neighbors.}.
See \cref{alg:admm} for the outline of the algorithm. 
\begin{algorithm}
	\caption{The ADMM method}
	\label{alg:admm}
	\begin{algorithmic}
		\STATE {\bfseries Input:} data $A$, regularization $c>0$, weights $\{\omega_{i,j}\}>0$.
		\STATE Initialize $X^{0}=\mathbf{0}, \Gamma^0=\mathbf{0}, Z^0=\mathbf{0}$, $s=0$.
		\STATE Set positive tuning constant $\nu>0$.
		\REPEAT
		\STATE $X^{s+1}=\arg\min_{X}L_\nu(X,Z^s,\Lambda^s)$
		\STATE $Z^{s+1}=\arg\min_{Z}L_\nu(X^{s+1},Z,\Lambda^s)$
		\STATE $\Lambda_{l}^{s+1}= \Lambda^{s}_{l}+\nu(V_{l}-X^{s+1}_{l_1}+X^{s+1}_{l_2})$
		\STATE $s=s+1$
		\UNTIL{$\|X^{s}-X^{s-1}\|_2\leq 10^{-4}$ is satisfied}
		\STATE {\bfseries Return:} $X$ and $Z$
	\end{algorithmic}
\end{algorithm}
The algorithm stops when iteration increment  is less than $10^{-4}$.
Once we obtain the solution $\widehat X$ to \cref{eqn:weightconvex} via  \cref{alg:admm}, the clusters are constructed using $\widehat X$ as the replacements of the original ones.
Two observations, $\widehat X_i,\widehat X_j$, belong to the same cluster, as long as  $\|\widehat X_i-\widehat X_j\|_2$ is below a pre-defined threshold (\eg $10^{-8}$).  As the number of clusters monotonically decreases when the value of regularization parameter $c$ increases  \cite{hocking2011clusterpath}, the number of the clusters is set by using different values of  $c$.

\section{Experiments}
\label{sec:experiments}
In this section, the performance of the weighted TV convex clustering model~\ref{eqn:weightconvex} is evaluated in both synthesized and real data sets.

\subsection{Experimental evaluation on synthesized data}
In this section, the performance of the model \cref{eqn:weightconvex} is evaluated when it is used  on the synthesized datasets that do not
satisfy the exact clustering condition stated in \cref{theorem_twobody} and \cref{theorem1}. As model \cref{eqn:weightconvex} can be viewed
as a convex relaxation of the $k$-means method, we focus on the comparison between the model \cref{eqn:weightconvex} and  two popular  implementations
of the $k$-means method, namely Lloyd's algorithm and \emph{$k$-means$++$}.

\subsubsection{Clustering data drawn from mixture of Gaussian.} Dataset of Mixture of Gaussian  is generated by randomly sampling several multivariate normal distributions~(MVN)  $\mathcal{N}(\mu_{i},\sigma^2 I_{n\times n})$, where  $\mu_i$ denotes the mean and $\sigma^2$ denotes the variance. In this experiments, {the dimension of points $n$ is $100$ and the dataset size $m$ is $30$.} Three clusters  corresponding to different distributions with the same cluster size
are generated with $\vec{\bf{\mu}}_1=0\vec{\bf1}_n$, $\vec{\bf\mu}_2= 3\vec{\bf1}_n$, $\vec{\bf\mu}_3= -3 \vec{\bf1}_n$,
where $\vec{\bf1}_n$ denotes the all-ones vector. Two values of variance are tested: $\sigma=1,2$.
The parameter $r$ is set to $r=0.02(2\sigma^2+5\sigma)$. See  \cref{fig:gaussian} for the average Rand index over $100$ runs.
It can be seen that model \cref{eqn:weightconvex} are slightly better in terms of accuracy and noticeably more stable than two $k$-means algorithms.

\begin{figure}
	\centering
	\subfloat[$\sigma=1$]{\label{fig:circle_input}{\includegraphics[width=0.5\textwidth]{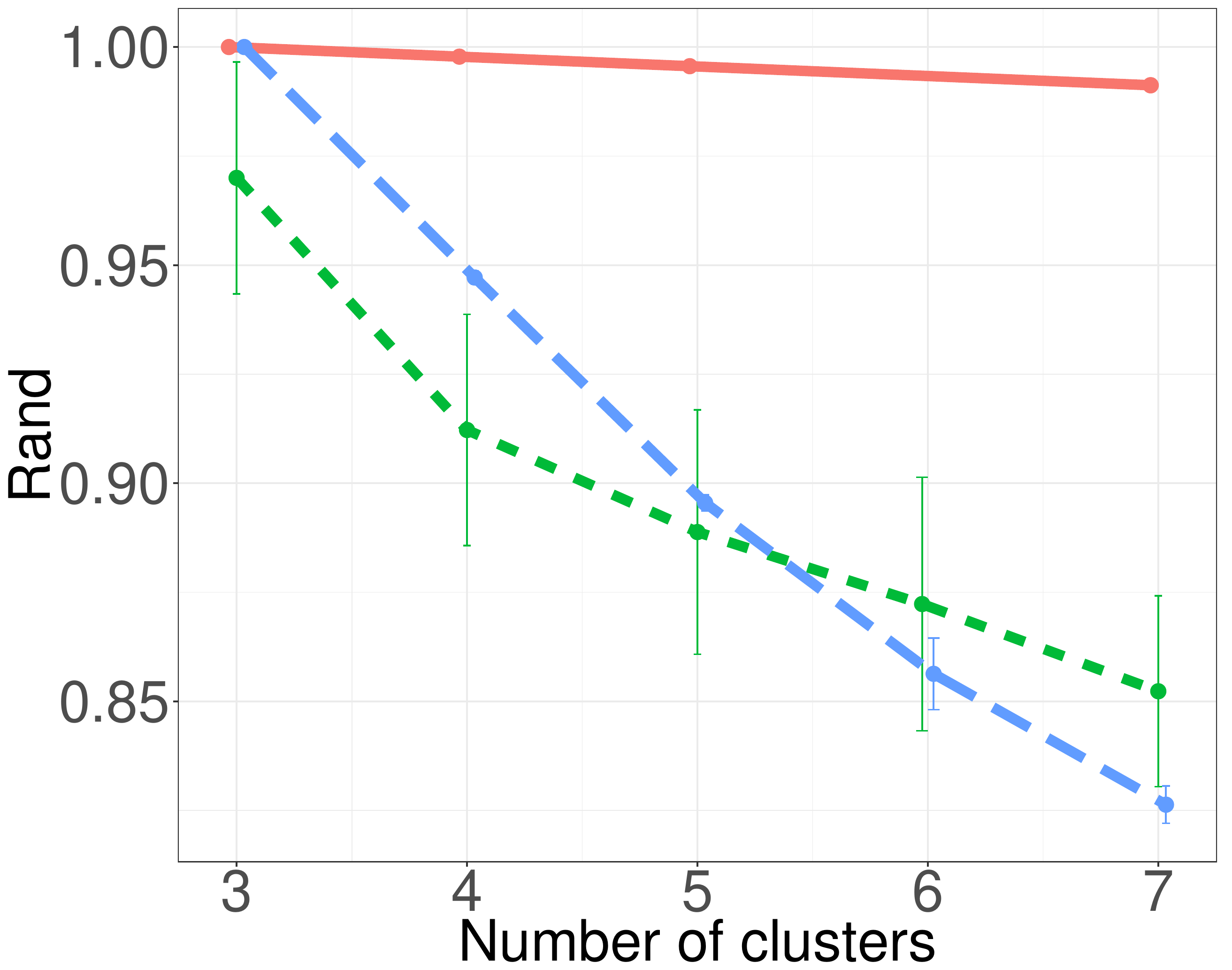}}}
	\subfloat[$\sigma=2$]{\label{fig:circle_kmeans}{\includegraphics[width=0.5\textwidth]{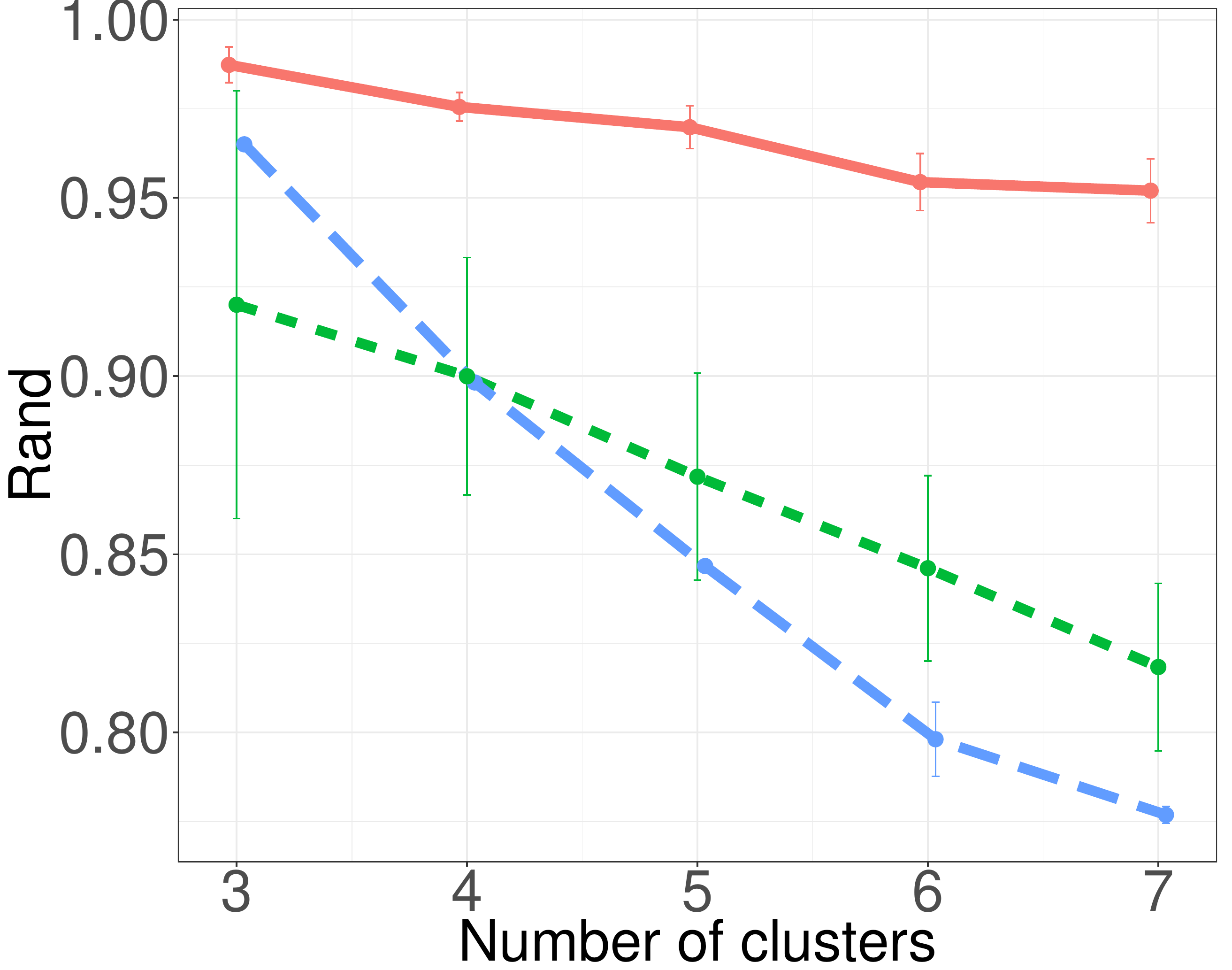}}}
	\caption{Comparison of average Rand index of different methods for clustering mixture of Gaussians over 200 runs. The correspondence
			between plot and method is: Model \cref{eqn:weightconvex}~({\protect\hwplotA}), Lloyd's algorithm~({\protect\hwplotH}), $k$-means++~({\protect\hwplotG}).}
	\label{fig:gaussian}
\end{figure}

\noindent{\bf Clustering non-convex clusters}.\quad
In this experiment,  the performance of the proposed model 
is evaluated on a classic two-body non-convex clustering problem: two embedded circle clusters~\cite{ng2001spectral}. See \cref{fig:results_circle}(a) for the visualization of such two clusters. It is  well-known  that the
$k$-means methods do not perform well on  such a dataset, as it is about minimizing the cluster variances (see \eg \cite{xu2005survey}.
The dataset is generated as follows.
One cluster (blue points) is given by  randomly sampling $250$ points from a $2$-dimensional  MVN $\mathcal{N}(0,I_{2\times 2})$ under Cartesian coordinate system. The other cluster (red
points) is generated by  randomly sampling $250$ points under polar coordinate system, in which radials follow normal distribution $\mathcal{N}(5,0.25^2)$ and angulars follow uniform distribution $\mathcal{U}(0,2\pi)$.
	
As the $k$-means method is sensitive to initialization, the mean and standard deviation~(SD) of the two $k$-means methods are reported over 100 randomly chosen initializations.    In the experiment, we repeated $k$-means algorithm $100$ times, and calculated the mean and standard deviation~(SD) of all Rand index. See \cref{tab:example} for the comparisons among three methods, and see \cref{fig:results_circle} for the illustration of one instance. The result from the model \cref{eqn:weightconvex} is much better than that from two $k$-means algorithms.

\begin{table}
	\caption{\label{tab:example}Rand indexes on clustering non-convex clusters }
	\centering
	\begin{tabular}{|l|*{5}{c|}}\hline
			Method & Lloyd's alg. & $k$-means++ & Model \cref{eqn:weightconvex} \\ \hline
			Mean & 0.517 &0.491& 0.996 \\ \hline
			SD & 0.012 & 0.010 & { 0}  \\ \hline
	\end{tabular}
\end{table}

\begin{figure}
		\centering
		\subfloat[input]{{\includegraphics[width=0.4\textwidth]{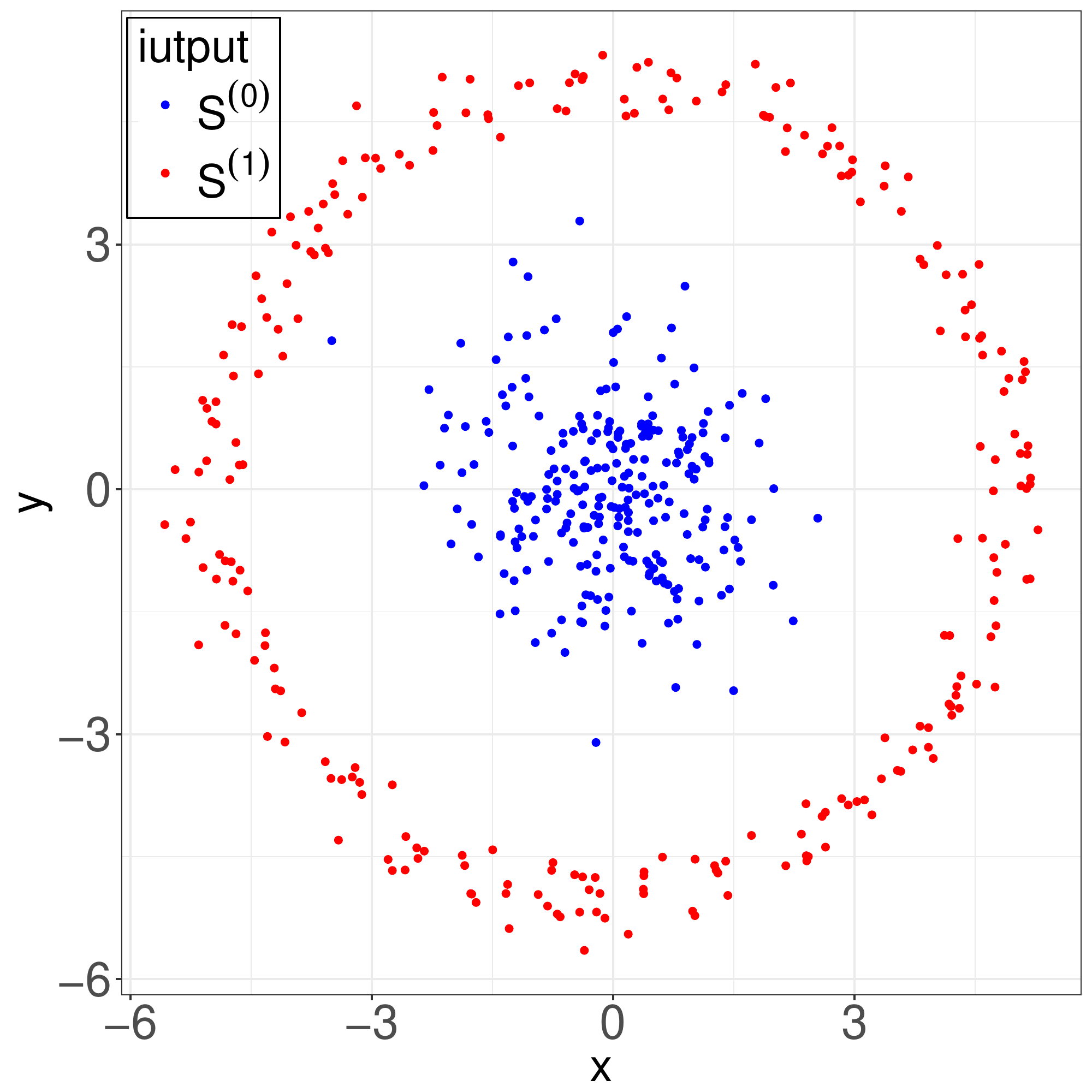}}}
		\subfloat[Lloyd's algorithm]{{\includegraphics[width=0.4\textwidth]{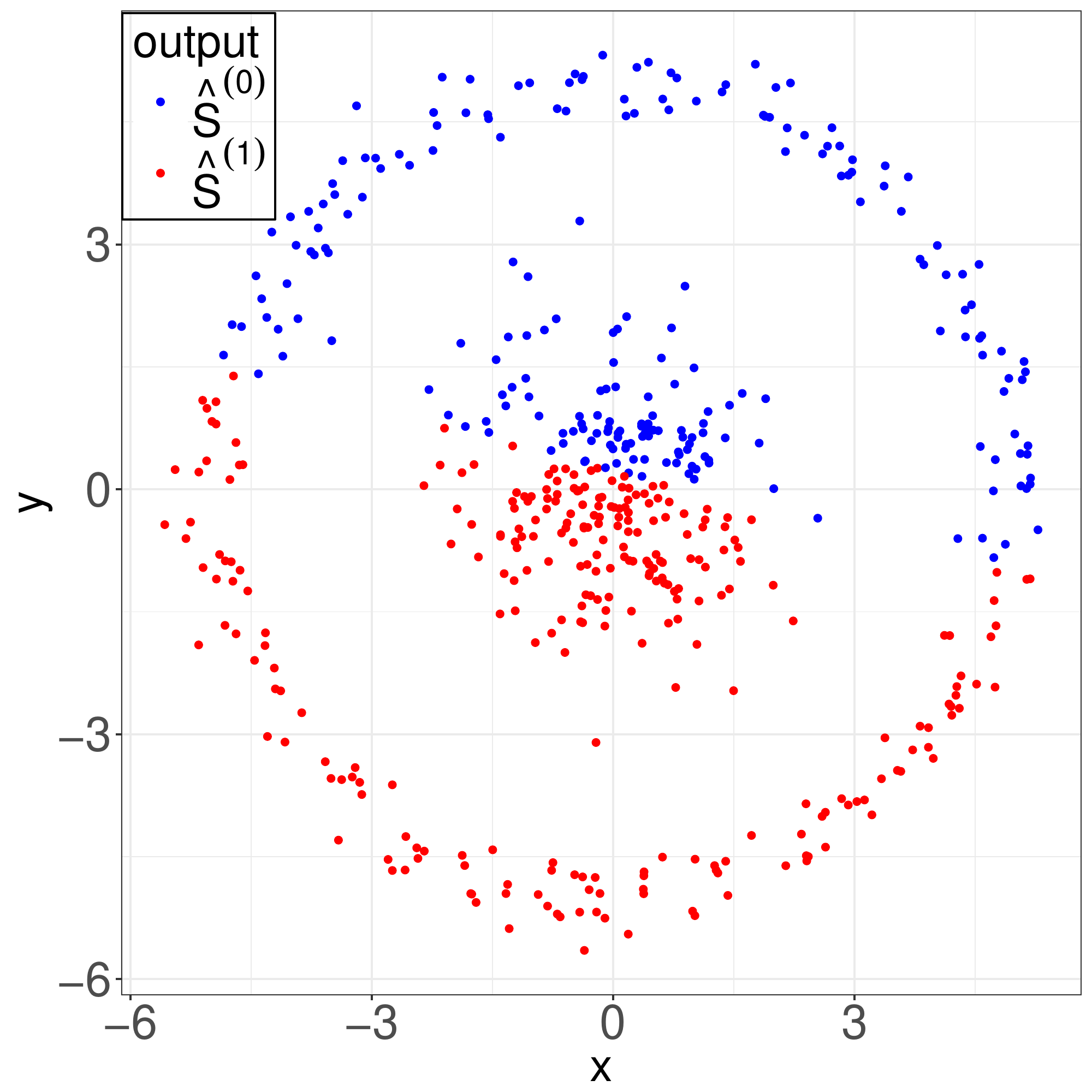}}} \\
		\subfloat[$k$-means++]{{\includegraphics[width=0.4\textwidth]{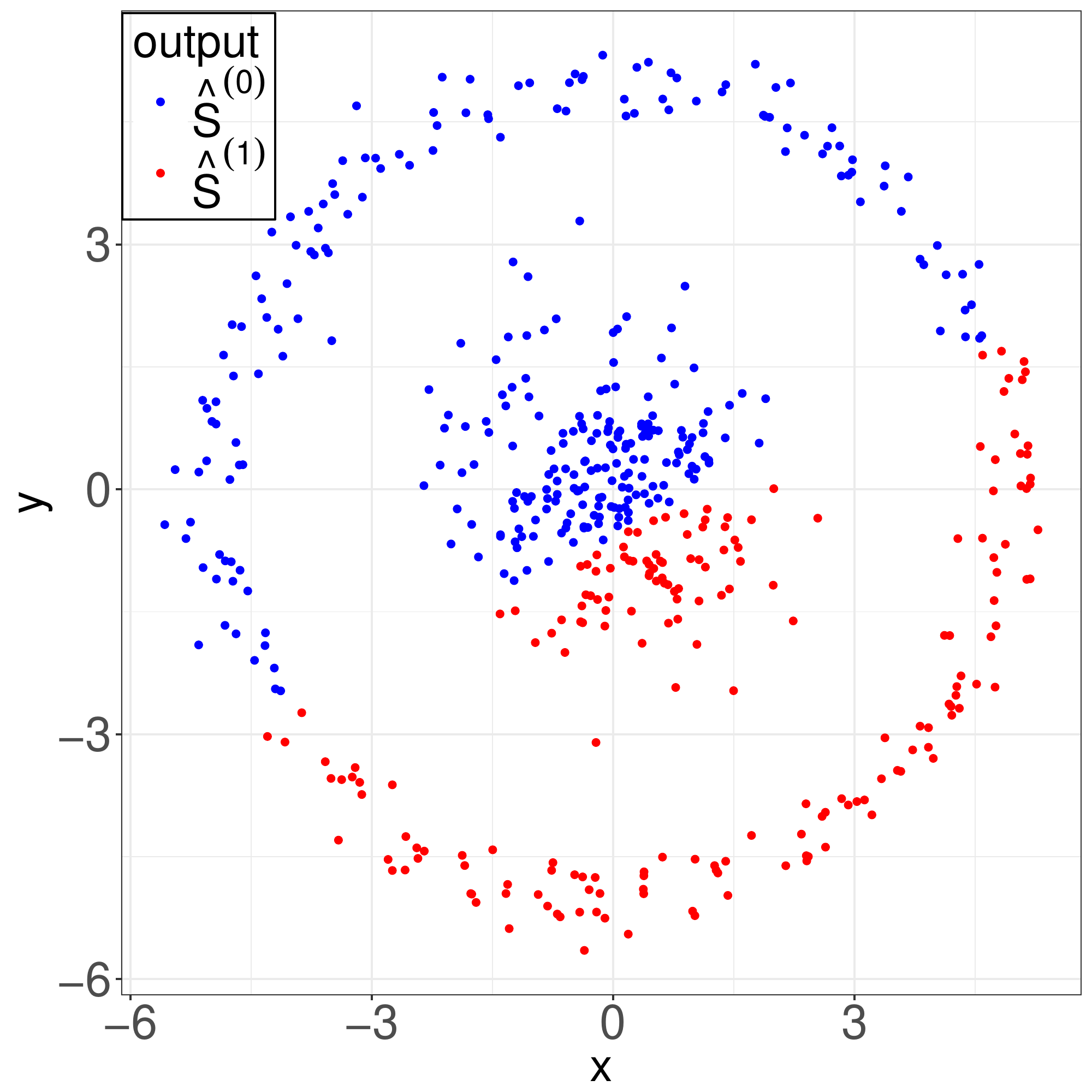}}}
		\subfloat[model \cref{eqn:weightconvex}]{{\includegraphics[width=0.4\textwidth]{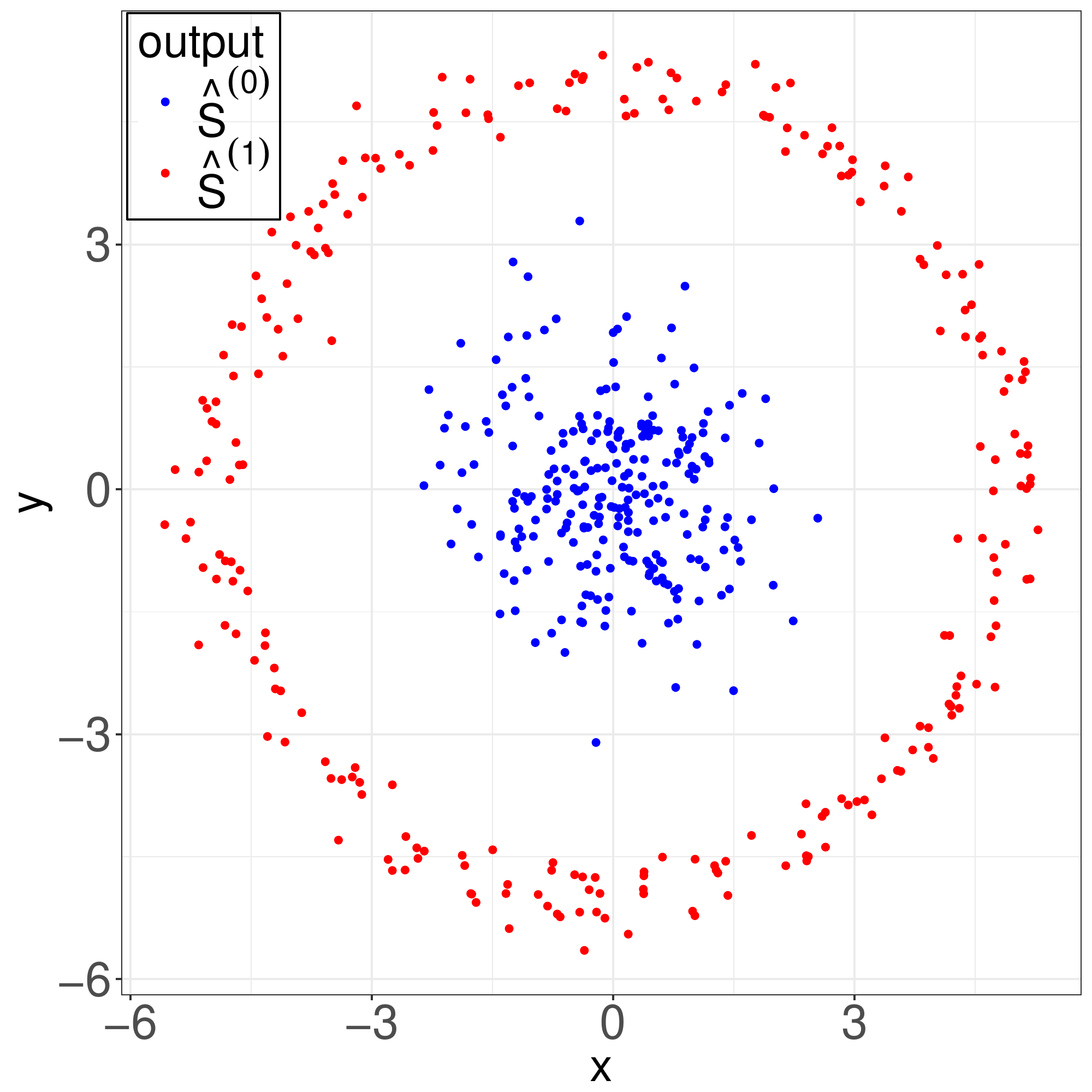}}}
		\caption{Illustrations of clustering non-convex clusters.}
		\label{fig:results_circle}
\end{figure}

\subsection{Experimental evaluation on real data}
In this section, we compared the clustering performance of model \cref{eqn:weightconvex} with several other methods on  several real data sets. The methods for comparison include two $k$-means algorithm: (1) Lloyd' $k$-means algorithm, (2) $k$-means++, and two hierarchical algorithms: (3) hierarchical single linkage (HC(Sin)), and (4) hierarchical average linkage (HC(Ave)). 
Through the experiments, the so-called \emph{Rand} index is used to evaluate clustering performance; see ~\cite{hubert1985comparing} for more details. It ranges between $0$ and $1$, and larger Rand index  implies better agreement with the true clusters.
The results from Model \cref{eqn:weightconvex} and two linkage methods
are reported from a single pass, and the results from two $k$-means algorithms are reported by taking the average of the outcomes over $100$ tests (each test takes the largest Rand index over $10$ randomly chosen initializations).

 Four datasets, from  UCI machine learning repository~\cite{Lichman:2013,martinez1998ar} are tested in the experiments
 
 \begin{enumerate}
 	\item "Iris": a classic testing dataset  for statistical classification, which contains 3 clusters with average 50 instances in $\BR^4$  each, where each cluster refers to a type of iris plant.
 	\item "AR face": it contains 2 clusters (male and female), each cluster contains about 2,000 color images corresponding to 126 people's faces. For computationally efficiency, each image is projected to a vector of length 100 by random matrix projection. The Rand index corresponding to AR face dataset is reported by taking the average of the outcomes over 50 tests (300 images are randomly drawn from each cluster for each test)
 	\item "Libras movement": it contains 15 classes of 24 instances in $\BR^{90}$ each, where each class references to a hand movement type in Libras. 
 	\item "Leaf": it contains 36 clusters with average 10 instances in $\BR^{14}$ each, where each cluster refers to one kind of leaf specimen. 
 \end{enumerate}
 
See \cref{fig:results_1}--\cref{fig:results_2} for the comparison of Rand indexes of the five methods. It can be seen that, in general, Model \cref{eqn:weightconvex} is the top performer among all compared methods, and its performance is consistent on different datasets with different characters.
Particularly, it significantly outperforms the others on AR dataset, which has much larger cluster size and much higher dimension than the others.
It indicates that the convex relaxation \cref{eqn:weightconvex} of $k$-means method has its advantages when processing the high-dimensional dataset with large cluster size, which is known to be challenging for $k$-means algorithms (see e.g \cite{garcia1999robustness}).
 
 \begin{figure}
	\centering
	\subfloat[Iris]{{\includegraphics[width=0.5\textwidth]{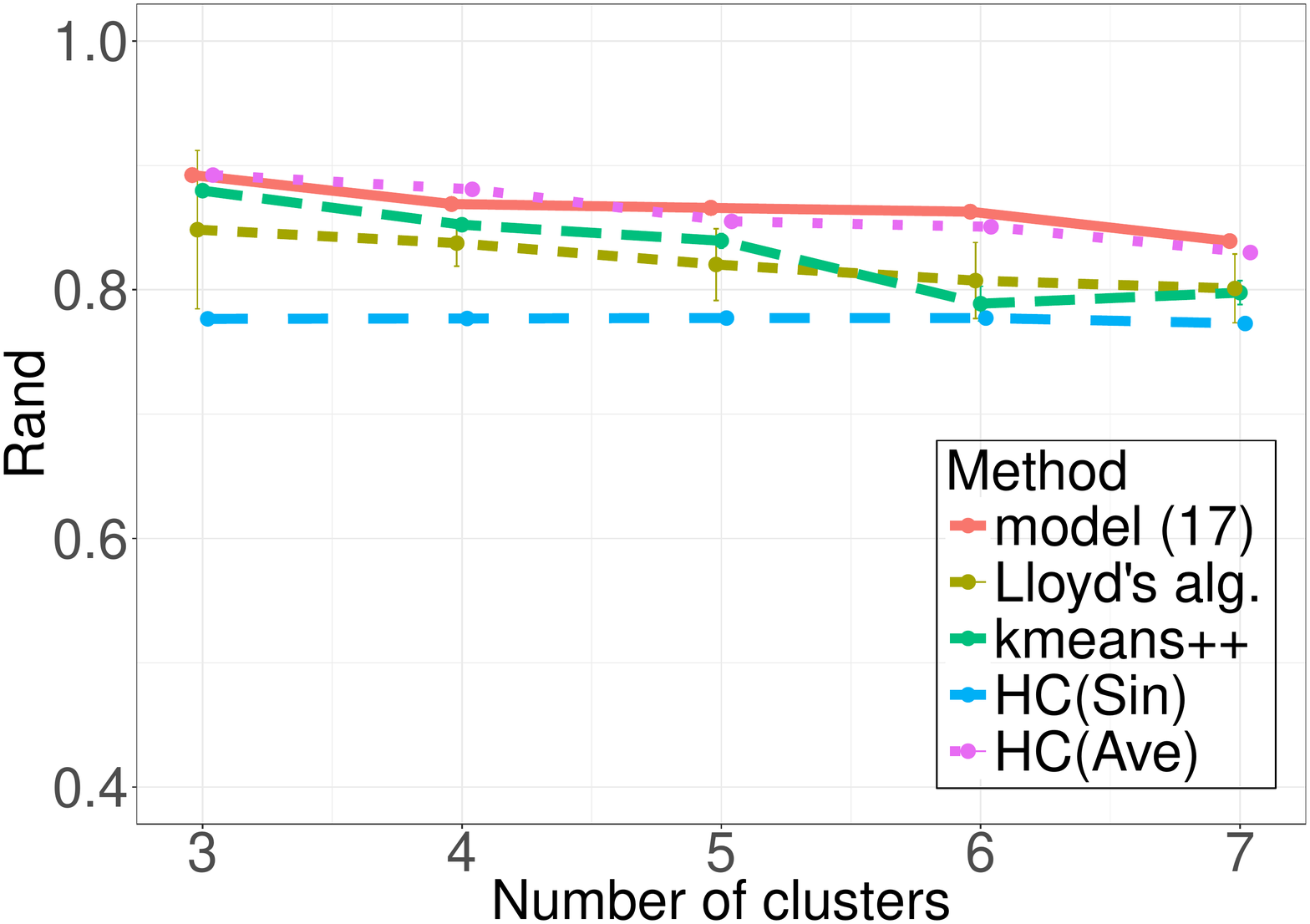}}}
	\subfloat[AR face]{{\includegraphics[width=0.5\textwidth]{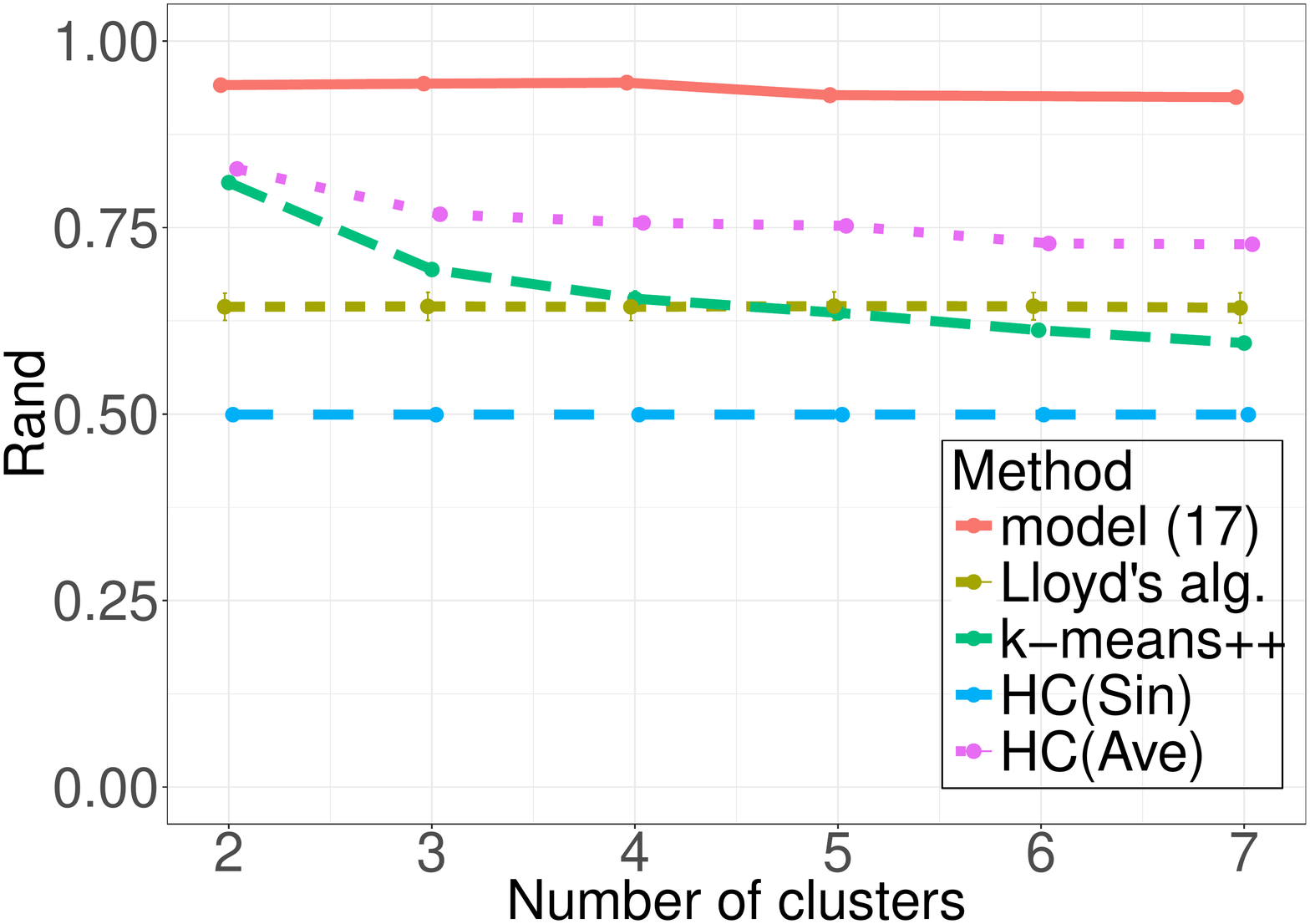}}}
	\caption{Comparison of Rand indexes among different methods on "iris" dataset and "AR" face dataset.
	}
	\label{fig:results_1}
\end{figure}

 \begin{figure}
	\centering
	\subfloat[Libras Movement]{{\includegraphics[width=0.5\textwidth]{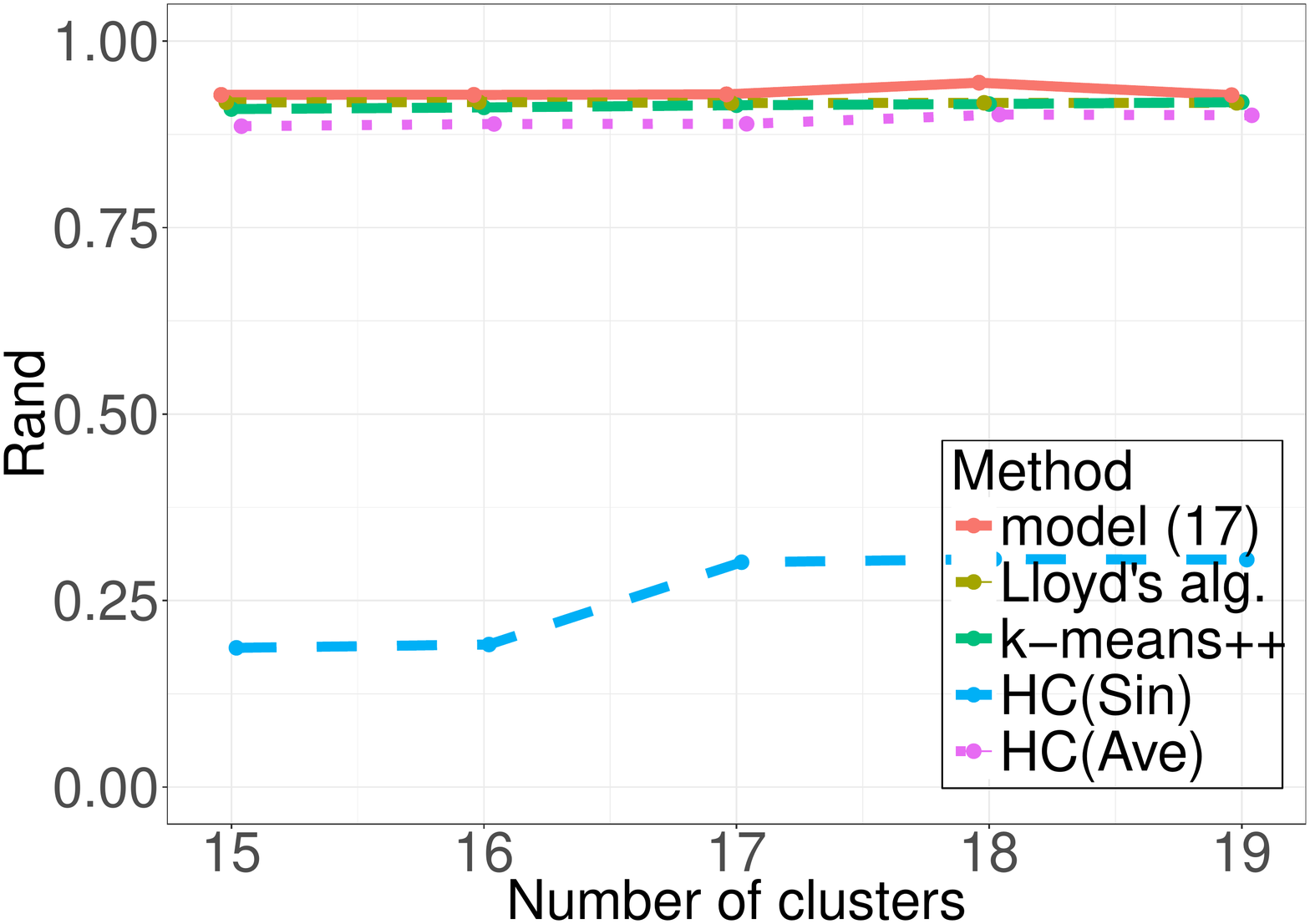}}}
	\subfloat[Leafs]{{\includegraphics[width=0.5\textwidth]{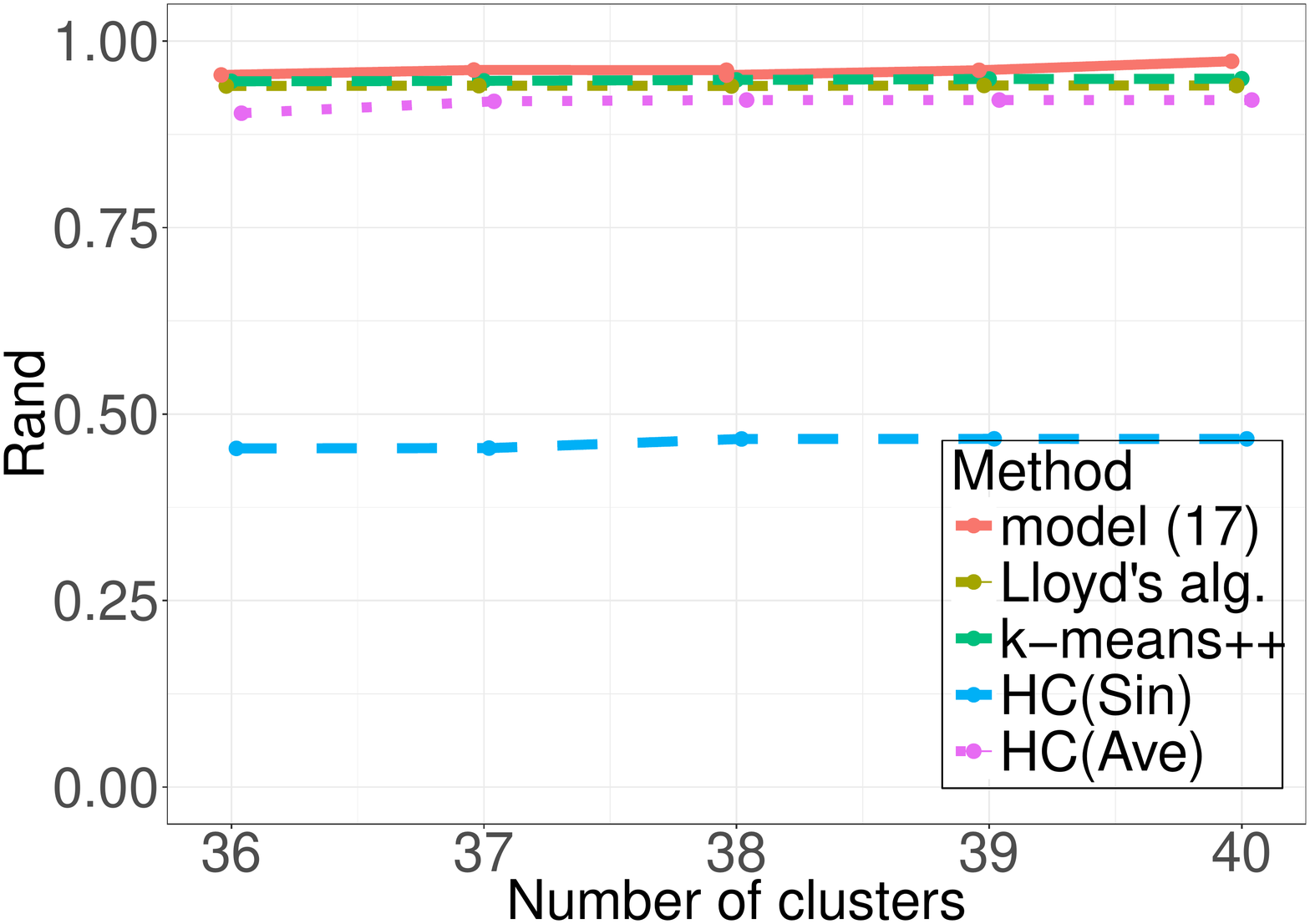}}}
	\caption{Comparison of Rand indexes among different methods on
		"Libras movement" dataset and "leafs" dataset}
	\label{fig:results_2}
\end{figure}

\section{Conclusion}
\label{sec:conclusion}
 This paper revisited TV based convex clustering by examining a weighted sum-of-$\ell_1$-norm relating convex model for data clustering. The   theoretical analysis on its exact clustering property goes beyond the existing ones to be applicable to the data with arbitrary number of clusters. Also, the separation condition on data for exact clustering guarantee is much shaper than that 
 of  the existing TV based convex clustering, and is also less restrictive than that for other convex clustering techniques, in terms of scalability and size ratios of clusters. 
   The experiments on real datasets also showed better performance of the proposed convex model over standard methods.
 The studies presented in this paper provide important insights and solid theoretical foundations for further development and studies of novel convex clustering methods.

\bibliographystyle{siamplain}
\bibliography{references_siam}

\appendix\label{appendix}
\section{Proof of \cref{theorem_twobody,theorem1}}
The proofs of \cref{theorem_twobody,theorem1} both take three main steps: (1) we derive an equivalent optimization problem of Model \cref{eqn:weightconvex} which is easier to be analysed owing to the variables separation inside $\ell_1$ norm; (2) the dual problem of the optimization is studied by constructing a dual pair for $k=2$ and $k\geq 3$ separately such that the minimizer of the optimization problem has some special properties; (3) based on the special properties derived in the previous step, we prove that  the minimizer of \cref{eqn:weightconvex} has exact clustering property.

	\subsection{Model equivalence and dual problem}
	Before proceeding, we  consider the column-centered version of $A$, denoted by $\tilde{A}$, which is defined as
	\begin{equation}
	\label{eqn:center}
	\tilde{A}=A-C,
	\end{equation}
	where $j$th column of $C$ is a constant vector with value $\frac{1}{m}\sum_{i=1}^mA_{ij}$. It can be seen that the mean of each column of $\tilde A$ is zero.
	
	 First of all, we prove that  Model \cref{eqn:weightconvex} taking $A$ as input is equivalent to Model \cref{eqn:weightconvex} taking $\tilde A$.
\begin{lemma}
		\label{thm:lemma1}
		Let $A\in \mathbb{R}^{m\times n}$ be a data matrix whose rows are arbitrarily sampled from $k$ sets $\mathcal{S}^{(s)}$ ($0\leq s\leq k-1$).
		Let $\tilde A$ denote its column-centered version defined by \eqref{eqn:center}. Then  the minimizer of the model \cref{eqn:weightconvex} has exact clustering property if and only if
		the minimizer of the following optimization problem 		\begin{equation}
		\label{eqn:eq1}
		\min_{\tilde{X}}\|\tilde{A}-\tilde{X}\|^2_F+\sum_{j=2}^m\sum_{i<j}ce^{-r\|\tilde{A}_i-\tilde{A}_j\|_2^2}\|\tilde{X}_i-\tilde{X}_j\|_1
		\end{equation}
		also has exact clustering property.
	\end{lemma}

\begin{proof}
		Suppose $\widehat{\tilde{X}}$ is the minimizer of Problem~\eqref{eqn:eq1}, then for any $X\in\mathbb{R}^{m\times n}$, we have
		\begin{equation*}
		\label{eqn:proof1}
		\begin{split}
		&\|A-X\|^2_F+\sum_{i<j}ce^{-r\|A_i-A_{j}\|_2^2}\|X_i-X_j\|_1\\
		=&\|A-C-(X-C)\|^2_F+\sum_{j=2}^m\sum_{i<j}ce^{-r\|A_i-C_i-(A_j-C_j)\|_2^2}\|X_i-C_i-(X_j-C_j)\|_1\\
		=&\|\tilde{A}-(X-C)\|^2_F+\sum_{j=2}^m\sum_{i<j}ce^{-r\|\tilde{A}_i-\tilde{A}_j\|_2^2}\|X_i-C_i-(X_j-C_j)\|_1\\
		\geq &\|\tilde{A}-\widehat{\tilde{X}}\|^2_F+\sum_{j=2}^m\sum_{i<j}ce^{-r\|\tilde{A}_i-\tilde{A}_{j}\|_2^2}\|\widehat{\tilde{X}}_i-\widehat{\tilde{X}}_j\|_1\\
		=&\|A-(\widehat{\tilde{X}}+C)\|^2_F+\sum_{j=2}^m\sum_{i<j}ce^{-r\|A_i-A_{j}\|_2^2}\|(\widehat{\tilde{X}}_i+C_i)-(\widehat{\tilde{X}}_j+C_j)\|_1.\\
		\end{split}
		\end{equation*}
		Thus, we have $\widehat{\tilde{X}}+C$ as the minimizer of Problem \cref{eqn:weightconvex}. Similarly, if $\hat{X}$ is a minimizer of Problem \cref{eqn:weightconvex}, then $\hat{X}-C$ is the minimizer of Problem \cref{eqn:eq1}.
		Furthermore, suppose $\widehat{\tilde{X}}$ is the minimizer of Problem~\eqref{eqn:eq1}, and it has exact clustering property, then $\widehat{\tilde{X}}_i-\widehat{\tilde{X}}_j=\widehat{\tilde{X}}_i-C_i-(\widehat{\tilde{X}}_j-C_j)$ implies that the minimizer of Problem \cref{eqn:weightconvex} has exact clustering property as well. Similarly, if the minimizer of Problem \cref{eqn:weightconvex} has exact clustering property, then  the minimizer of Problem \cref{eqn:eq1} also has exact clustering property.
	\end{proof}
Therefore, in the remaining of this paper, we will assume that $A$ is column-centred such that the mean of each column of $A$ is zero.

For convenience, set $B=DA$. Using the same arguments as  \cite[Lemma 1, Lemma 2]{zhu2014convex}, we have
\begin{lemma}
	\label{thm:lemma2}
	Given any column centered-matrix $A$, $\hat{X}$ is minimizer of Problem \cref{eqn:weightconvex}, if and only if $D\hat{X}$ is the minimizer of the following problem,
	\begin{equation}
	\label{formula5}
	\hat{Y}=\argmin_{Y\in {\mathcal{D}}}\frac{1}{m}\|B-Y\|_F^2+c\sum_{i=1}^{{ m \choose 2}}\gamma_{i}\|Y_i\|_1,
	\end{equation}
	where $\gamma_i=\exp(-r\|(DA)_i\|_2^2) $, and $\mathcal{D}$ is the range space of $D$  defined by $\mathcal{D}=\{D X\ |\ X\in \mathbb{R}^{m\times n}\}$.
\end{lemma}

The problem \cref{eqn:weightconvex} is now simplified as the entries in sum-of-$\ell_1$-norm separable. To tackle the technical challenges  raised by  the introduced constraint $Y \in \mathcal{D}$,
we consider the dual problem of \cref{formula5}. By standard arguments in convex analysis \cite[Page 301]{bertsekas2009convex}, The variable pair $\{\hat{Y},\hat{\Lambda}\}$ is an optimal primal-dual pair if and only if
\begin{equation}\label{eqn:dual}
\hat{Y_i}=\operatornamewithlimits{argmin}_{Y_i^T\in \mathbb{R}^n}\frac{1}{m}\|B_i-Y_i\|_2^2+c\gamma_i\|Y_i\|_1-Y_i\Lambda_i^T,
\end{equation}
subject to $Y\in \mathcal{D},\ \Lambda^T\in \mathcal{D}^\perp$, where $ \mathcal{D}^\perp$ denotes the orthgonal complement of $\mathcal{D}$.
Hence, to prove $\hat{X}$ of having exact clustering property, it suffices to find some dual parameter $\widehat{\Lambda}^T \in \mathcal{D}^\perp$, such that the corresponding primal minimizer $\hat{Y}\in \mathcal{D}$  has the following form,
\begin{equation}
\label{eqn:solutionform}
\hat{Y}_{p}=
\left\{ \begin{array}{l}
\vec{\textbf{0}},\ \ \ \ \ \makebox{if}\ p\in\mathcal{P}^{(0)},\\
\vec{\textbf{e}}^{s,l},\ \ \makebox{if}\ p\in\mathcal{P}^{(s,l)},
\end{array} \right.
\end{equation}
where $(\vec{\textbf{e}}^{s,l})^T\in \mathbb{R}^n$ is some non-zero constant vector. The index set $\mathcal{P}^{(0)}$ is defined in \cref{eqn:P0} and
\begin{equation}
\label{eqn:Pkl}\mathcal{P}^{(s,l)}=\{(i-1)(m-1)+(j-i)\ |\  i\in \mathcal{I}^{(s)},\ j\in \mathcal{I}^{(l)}\}
\end{equation} represents the row indexes of $DX$ which  are corresponding to the differences of two rows from different clusters $\mathcal{S}^{(s)}$ and $\mathcal{S}^{(l)}$, where $0\leq s<l\leq k-1$. $\mathcal{P}^{(1)}=\cup_{i<j}\mathcal{P}^{(i,j)}$ defined in \cref{eqn:P1} is
the index set of all differences from distinct clusters.

\subsection{Proof of \cref{theorem_twobody}}
\label{sec:prf of two-cluster}
First of all, we show some sufficient condition for primal-dual solver.
\begin{theorem}[sufficient condition for $2$-cluster]
\label{thm:suf_two}
 Assume $c,r\in \mathbb{R}^+$ and $\widehat{\Lambda}^T\in \mathcal{D}^\perp$.  $\vec{\textbf{e}}^T\in \mathbb{R}^n$ is some nonzero column vector. $\mathcal{P}^{(0)}$ and $\mathcal{P}^{(1)}$ are defined in \cref{eqn:P0} and \cref{eqn:P1}.
 
	Suppose for $p\in\mathcal{P}^{(1)}$, $\hat\Lambda_{pq}$ satisfies
	\begin{subequations}
	\begin{align}
	&c\gamma_{p}\leq \frac{2}{m}\left|B_{pq}+\frac{m}{2}\widehat{\Lambda}_{pq}\right|,\label{eqn:cond_two10}\\
	&\left(B_{pq}+\frac{m}{2}\widehat{\Lambda}_{pq}\right)-\frac{mc\gamma_p}{2}\sign\left(B_{pq}+\frac{m}{2}\widehat{\Lambda}_{pq}\right)=\vec{\textbf{e}}_q,\label{eqn:cond_two11}
	\end{align}
	\end{subequations}
		
		 and for $p\in\mathcal{P}^{(0)}$, we have
		\begin{equation}
		\label{eqn:cond_two2}
		c\gamma_p>\frac{2}{m}\left|B_{pq}+\frac{m}{2}\widehat{\Lambda}_{pq}\right|,
		\end{equation}
$\widehat{\Lambda}$ is the optimal dual variable in \cref{eqn:dual}, and $\hat{Y}$ can be derived respectively. Together with the fact that $\widehat{Y}=D\widehat{X}$, $\widehat{X}$ satisfies the exact clustering property.
\end{theorem}
\begin{proof}
If $\widehat{\Lambda}$ is  the optimal dual variable in \cref{eqn:dual}, the optimal primal variable $\widehat{Y}$ can be derived as 
\begin{equation}\label{eqn:proof2}
	\widehat{Y}_{pq}=\left(B_{pq}+\frac{m}{2}\widehat{\Lambda}_{pq}\right)-\frac{mc \gamma_{p}}{2}\sign\left(B_{pq}+\frac{m}{2}\widehat{\Lambda}_{pq}\right),
	\end{equation}
	if $c\gamma_{p}\leq \frac{2}{m}|B_{pq}+\frac{m}{2}\widehat{\Lambda}_{pq}|$ and $\widehat{Y}_{pq}=0$ otherwise, for  $p=1,\dots,\binom{m}{2}$ and  $q=1,\dots,n$.

By the conditions \cref{eqn:cond_two10,eqn:cond_two10,eqn:cond_two11,eqn:cond_two2},we can see that 
\begin{equation}
	\label{eqn:solutionform_two}
	\widehat{Y}_{p}=
	\left\{ \begin{array}{l}
	\vec{\textbf{0}},\ \ \mbox{if}\ p\in\mathcal{P}^{(0)},\\
	\vec{\textbf{e}},\ \ \mbox{if}\ p\in\mathcal{P}^{(1)},
	\end{array} \right.
	\end{equation}
which meets the primal minimizer form in \cref{eqn:solutionform}. Therefore, $(\widehat{Y},\widehat{\Lambda})$ is the optimal primal-dual pair in \cref{eqn:dual}.
\end{proof}

Under the sufficient condition in \cref{thm:suf_two}, we have $\widehat{X}$ has exact clustering property with
	\[\widehat{X}_{i,j}=
	\left\{ \begin{array}{l}
	\frac{m_1}{m_0+m_1}\vec{\textbf{e}}_j, \quad i\in\mathcal{I}^{(0)},\\
	\frac{-m_0}{m_0+m_1}\vec{\textbf{e}}_j, \quad i\in\mathcal{I}^{(1)}.
	\end{array} \right.
	\]

We explicitly construct the optimal primal and dual pair $(\widehat{Y};\widehat{\Lambda})$ of Model
	\cref{formula5}.
The sufficient conditions can be further simplified by  \cref{lem:sufficient_two1,lem:sufficient_two2}. Before further relaxation of the conditions, we introduce some notations for simplicity. Here we define
\begin{equation}
\label{eqn:notation_def_two}
	\left\{
	\begin{array}{ll}
	\epsilon_i=\frac{8m_{1-i}(m_i-1)+4m_i^2}{mm_i^2}\\
	\mu_{pq}=B_{pq}+\frac{m}{2}\widehat{\Lambda}_{pq}\\
	\tau_q=\frac{1}{m_0m_1}\sum_{i\in\mathcal{P}^{(1)}}B_{iq}\\
	\rho=\frac{1}{m_0m_1}\sum_{i\in\mathcal{P}^{(1)}}\gamma_{i}
	\end{array}
	\right.
\end{equation}
	where $1\leq p \leq {m\choose 2}$ and $1\leq q\leq n$.
	
\begin{lemma}
\label{lem:sufficient_two1}
Denote $\mathcal{Q}=\{q|\tau_q=0, 1\leq q \leq n\}$. 
	Suppose there exist $c,r\in \mathbb{R}^+$ such that
	\begin{equation}
	\label{eqn:ec_two}
	|\tau_q|>\max_{p\in \mathcal{P}^{(1)}}\left|\frac{mc}{2}(\rho-\gamma_p)\right|,
	\end{equation}
	and
	\begin{equation}
	\label{eqn:cond_two102}
	\frac{2|\tau_q|}{m\rho}>c,
	\end{equation}
	for any $q\not\in\mathcal{Q}$.
Then there exist proper $\widehat\Lambda_{pq}$ ($p\in \mathcal{P}^{(1)},\ 1\leq q\leq n$) to  meet the condition \cref{eqn:cond_two10,eqn:cond_two11}.
\end{lemma} 
\begin{proof}
Since the variable $\vec{\textbf{e}}_q$ is a constant independent from $p$, summing both sides of~\eqref{eqn:cond_two11} for $p\in \mathcal{P}^{(1)}$, we have
	\begin{equation}
	\label{eqn:relax3_two}
	\mu_{pq}=\tau_q+\frac{m}{2m_1m_0}\sum_{i\in\mathcal{P}^{(1)}}\widehat{\Lambda}_{iq}-\frac{mc}{2}\left[\frac{1}{m_0m_1}\sum_{i\in\mathcal{P}^{(1)}}\gamma_i \sign(\mu_{iq})-\gamma_{p}\sign(\mu_{pq})\right].
	\end{equation}
	Notice that $\widehat{\Lambda}^T\in\mathcal{D}^\perp$, we have $\sum_{i\in\mathcal{P}^{(1)}}\widehat{\Lambda}_{iq}=0$. Then, the equality \cref{eqn:relax3_two} can be further simplified as
	\begin{equation}
	\label{eqn:relax4_two}
	\mu_{pq}=\tau_q
	-\frac{mc}{2}\left[\frac{1}{m_0m_1}\sum_{i\in\mathcal{P}^{(1)}}\gamma_i \sign(\mu_{iq})-\gamma_{p}\sign(\mu_{pq})\right].
	\end{equation}
	
	When $q\not\in\mathcal{Q}$,
	if there exist $c,r\in \mathbb{R}^+$ such that
	\[
	|\tau_q|>\max_p\left|\frac{mc}{2}(\rho-\gamma_p)\right|,
	\]
it is easy to see that $\sign(\mu_{pq})=\sign(\tau_q),$
	and \cref{eqn:relax4_two} holds with
	\begin{equation}
	\label{eqn:lambdaf1_two}
	\begin{split}
	&\widehat{\Lambda}_{pq}
	=\frac{2}{m}(\tau_q-B_{pq})-c(\rho-\gamma_{p})\sign(\tau_q).\\
	\end{split}
	\end{equation}

If $q\in\mathcal{Q}$, the inequality~\eqref{eqn:relax4_two} holds by choosing
	\begin{equation}\label{eqn:lambdaf0_two}
	\widehat{\Lambda}_{pq}=-\frac{2}{m}B_{pq}.
	\end{equation}
Therefore, Whenever $\tau_q=0$ or not, $\widehat{\Lambda}_{pq}$ can be written in the form of \cref{eqn:lambdaf1_two}.

Substituting \cref{eqn:lambdaf1_two} into \cref{eqn:cond_two11}, we will have
	\[
	\vec{\mathbf{e}}_q=\tau_q-\frac{mc\rho}{2} \sign(\tau_q),
	\]
	which is independent from index $p$. In other words, the inequality \cref{eqn:ec_two} implies \cref{eqn:cond_two11}.
	
	Secondly, provided that the inequality~\eqref{eqn:ec_two} holds true, we have
	\begin{equation}
	\label{eqn:lambdaf_two}
	\begin{split}
	\mu_{pq}
	=\tau_q-\frac{mc}{2}(\rho-\gamma_{p}){\sign{(\tau_q)}}.
	\end{split}
	\end{equation}
	Then, replacing the term $\mu_{pq}$ in \cref{eqn:cond_two10} by \cref{eqn:lambdaf_two} gives
	\[
	c\gamma_p\leq \frac{2}{m}\left|\tau_q\\
	-\frac{mc}{2}(\rho
	-\gamma_{p})\sign{(\tau_q)}\right|,
	\]
	for any $p\in \mathcal{P}^{(1)}$ and $q\in \mathcal{Q}$.
	
	If 
	\[
	\frac{2|\tau_q|}{m\rho}>c,
	\]
	we have 
	\[
	c\gamma_p\leq\frac{2}{m}|\tau_q|+c\gamma_p-c\rho
	\leq \frac{2}{m}\left||\tau_q|-\frac{mc}{2}(\rho-\gamma_p)\right|
	= \frac{2}{m}\left|\tau_q-\frac{mc}{2}(\rho-\gamma_{p})\sign{(\tau_q)}\right|.
\]
Thus \cref{eqn:ec_two,eqn:cond_two102} imply \cref{eqn:cond_two10,eqn:cond_two11}.
\end{proof}
\begin{lemma}
\label{lem:sufficient_two2}
Assume $\{\widehat{\Lambda}_{pq}\}_{p\in\mathcal{P}^{(1)},1\leq q\leq n}$ satisfies \cref{eqn:cond_two10,eqn:cond_two11}. If
\begin{equation}
	\label{eqn:relaxc2_two}
	\begin{split}
	&\max_{0\leq i\leq 1}\left[{{\epsilon_i}\dia(\mathcal{S}^{(i)})}+\frac{4c m_{1-i}}{m_i}\max_{j\in\mathcal{P}^{(1)}}(\gamma_j)\right]<c\gamma_p,
	\end{split}
	\end{equation}
	for any $p\in \mathcal{P}^{(0)}$. Then there exist $c,r\in\mathbb{R}^+$ and $\widehat{\Lambda}\in\mathcal{D}^\perp$ satisfying Condition \cref{eqn:cond_two2}.
\end{lemma}
\begin{proof}
We will establish that if $\{\widehat{\Lambda}_{pq}\}_{p\in\mathcal{P}^{(1)}}$ satisfies \cref{eqn:cond_two10,eqn:cond_two11}, then there exists $\{\widehat{\Lambda}_{pq}\}_{p\in\mathcal{P}^{(0)}}$,  such that $\widehat{\Lambda}$ belongs to $\mathcal{D}^\perp$, and
		\begin{equation}
		\label{eqn:compare_two}
		\begin{split}
		|\mu_{pq}|\leq& \max\limits_{0\leq i\leq 1}
		\left[\frac{m\epsilon_i\dia(\mathcal{S}^{(i)})}{2}+\frac{2c mm_{1-i}}{m_i}\max_{j\in\mathcal{P}^{(1)}}(\gamma_j)\right]
		\end{split}
		\end{equation}
		for any $ p\in \mathcal{P}^{(0)}$. Combined \cref{eqn:compare_two} with Condition \cref{eqn:relaxc2_two}, we can easily get the conclusion. 
		
		Now we focus on constructing $\hat{\Lambda}\in \mathcal{D}^\perp$ to meet \cref{eqn:compare_two}.
		
		Since $Y\in \mathcal{D}$ if and only if $HY=0$, $\mathcal{D}^\perp$ is exactly the row space of $H$, where $H=(D^{m-1},I)$. Then by direct calculation, ${\widehat{\Lambda}}^T\in \mathcal{D}^\perp$ is equivalent to,
		
		\begin{equation}\label{formula25_two}
		\left( \begin{array}{cc}
		-I_{(m_0-1)\times(m_0-1)},(D^{m_0-1})^T
		\end{array} \right){\hat\Lambda}_{[{\mathcal{P}}_0^{(0)},q]}
		=-\left( \begin{array}{cccc}
		\underset{p\in\mathcal{P}^{(1)}_2}\sum{\hat\Lambda}_{pq},
		\underset{p\in\mathcal{P}^{(1)}_3}\sum{\hat\Lambda}_{pq},
		\dots,
		\underset{p\in\mathcal{P}^{(1)}_{m_0}}\sum{\hat\Lambda}_{pq}
		\end{array} \right)^T
		\end{equation}\\
		and
		\begin{equation}
		\label{formula26_two}
		(D^{m_1})^T{\widehat{\Lambda}}_{[{\mathcal{P}}_1^{(0)},q]}=\sum_{j=1}^{m_0}{\widehat{\Lambda}}_{[{\mathcal{P}}^{(1)}_j,q]},
		\end{equation}
where
\[
{\mathcal{P}}^{(0)}_j=\{(s-1)(m-1)+(l-s)\ |\ s,l\in \mathcal{I}^{(j)}\ with\ s<l\}
\] 
and
\[
{\mathcal{P}}^{(1)}_j=\{(j-1)(m-1)+(i-j)\ |\ i\in \mathcal{I}^{(1)}\}
\] is the subset of $\mathcal{P}^{(1)}$, with $\cup_{l=1}^{m_0}\mathcal{P}^{(1)}_l=\mathcal{P}^{(1)}$. ${\widehat{\Lambda}}_{[{\mathcal{P}}_i^{(j)},q]}$ denotes the column vector whose elements are $\{\widehat{\Lambda}_{pq}\}_{p\in \mathcal{P}_i^{(j)}}$.
		The summation of all elements in the left side of  \cref{formula26_two} is zero, so there exists $\widehat{\Lambda}$ satisfying \cref{formula25_two} and \cref{formula26_two} only if $\sum_{p\in\mathcal{P}^{(1)}}\hat{\Lambda}_{pq}=0$.
		
		For $p\in\mathcal{P}^{(1)}$, by conditions \cref{eqn:cond_two10,eqn:cond_two11}, we have
		\[
	\widehat{\Lambda}_{pq}=\frac{2}{m}\tau_q-\frac{2}{m}B_{pq}-{c}(\rho-\gamma_{p}){\sign(\tau_q)}.
		\]
		Then
		\begin{eqnarray*}
		\frac{m}{2}{\sum_{p\in{\mathcal{P}}_t^{(1)}}\widehat{\Lambda}_{pq}}
		&=&\left({m_1}\tau_q-\sum_{p\in {\mathcal{P}}_t^{(1)}}B_{pq}\right)-\frac{mc}{2}\left(m_1\rho-\sum_{p\in\mathcal{P}^{(1)}_t}\gamma_p\right){\sign(\tau_q)}\\
		&=&-\frac{mcm_1}{2}\left(\rho-\frac{1}{m_1}\sum_{p\in\mathcal{P}^{(1)}_t}\gamma_p\right){\sign(\tau_q)}+\left(\frac{m_1}{m_0}\sum_{i\in\mathcal{I}^{(0)}}A_{iq}-m_1A_{tq}\right).
		\end{eqnarray*}
Obviously, $\sum_{t=1}^{m_0}{\sum_{p\in{\mathcal{P}}_t^{(1)}}\widehat{\Lambda}_{pq}}=0$ and
		\[
		\left\|\frac{m}{2}{\widehat{\Lambda}}_{[{\mathcal{P}}_t^{(1)},q]}\right\|_\infty\leq \frac{m_1(m_0-1)}{m_0}\dia(\mathcal{S}^{(0)})+{cm}{m_1}\max_{p\in \mathcal{P}^{(1)}}(\gamma_p)
		\]
		for any $t\in \mathcal{I}^{(0)}$.
		
	Considering the constraint 
	\begin{equation}
	\label{eqn:restrictConst}
	(D^{m_0})^T{\widehat{\Lambda}}_{[{\mathcal{P}}_0^{(0)},q]}=-\left( \begin{array}{cccc}
		\underset{p\in\mathcal{P}^{(1)}_1}\sum{\hat\Lambda}_{pq},
		\underset{p\in\mathcal{P}^{(1)}_2}\sum{\hat\Lambda}_{pq},
		\dots,
		\underset{p\in\mathcal{P}^{(1)}_{m_0}}\sum{\hat\Lambda}_{pq}
		\end{array} \right)^T,	\end{equation}
	we can see that  \cref{formula25_two} is the sub-constraint of \cref{eqn:restrictConst}.	Therefore, we need to find proper $\widehat\Lambda$ obeying \cref{eqn:restrictConst,formula26_two}.
	
	By \cite[Lemma 3]{zhu2014convex}, given $C_n\in \mathbb{R}^n$, \ie $C_n=(c_1,c_2,\dots,c_n)^T$ such that  $\sum_{i=1}^n c_i=0$, there exists $X\in \mathbb{R}^{C_n^2}$ with $\|X\|_{\infty}\leq\frac{2}{n}M$ and $(D^n)^TX=C_n$, where  $\|C_n\|_{\infty}\leq M$. 
So there exist $\widehat{\Lambda}_{pq}(\ p\in\mathcal{P}^{(0)}_0)$ satisfying \eqref{formula25_two} and
		\[
		\left|\frac{m}{2}{\widehat{\Lambda}}_{pq}\right|\leq
		\frac{2m_1(m_0-1)}{m_0^2}\dia(\mathcal{S}^{(0)})+\frac{2cmm_1}{m_0}\max_{i\in \mathcal{P}^{(1)}}(\gamma_i).
		\]
		Furthermore, we have
		\begin{equation}
		\label{eqn:I0constraint}
		\left|B_{pq}+\frac{m}{2}\widehat{\Lambda}_{pq}\right|\leq
		\frac{4m_1(m_0-1)+2m_0^2}{m_0^2}\dia(\mathcal{S}^{(0)})+\frac{2c mm_1}{m_0}\max_{i\in\mathcal{P}^{(1)}}(\gamma_i),
		\end{equation}
for $p\in \mathcal{P}_0^{(0)}$.

By the symmetry of $\widehat{\Lambda}_{pq}(p\in\mathcal{P}^{(0)}_0)$ and $\widehat{\Lambda}_{pq}(p\in\mathcal{P}^{(0)}_1)$, there exist $\widehat{\Lambda}_{pq}(p\in\mathcal{P}^{(0)}_1)$  satisfying \eqref{formula26_two}, and
		\begin{equation}
		\label{eqn:I1constraint}
		\left|B_{pq}+\frac{m}{2}\widehat{\Lambda}_{pq}\right|\leq
		\frac{4m_0(m_1-1)+2m_1^2}{m_1^2}\dia(\mathcal{S}^{(1)})
		+\frac{2c mm_0}{m_1}\max_{i\in\mathcal{P}^{(1)}}(\gamma_i),
		\end{equation}
for $p\in\mathcal{P}_1^{(0)}$.

Therefore, \cref{eqn:I0constraint,eqn:I1constraint} meets the condition \cref{eqn:compare_two}.
\end{proof}

Now, we begin to prove \cref{theorem_twobody}.
\begin{proof}[Proof of \cref{theorem_twobody}]
If there exist $c$ and $r$ such that they can guarantee the conditions \cref{eqn:ec_two,eqn:cond_two102,eqn:relaxc2_two} in \cref{lem:sufficient_two1,lem:sufficient_two2}, we have optimal solution satisfying exact clustering property. 

Here the condition \cref{eqn:relaxc2_two} is equivalent to 
\begin{equation}
	\label{eqn:condf3_two}
	\gamma_p-\frac{4 (m-m_i)}{m_i}\max_{j\in\mathcal{P}^{(1)}}(\gamma_j)>0
	\end{equation}
	and
	\begin{equation}
	\label{eqn:condf4_two}
	\frac{\epsilon_i\dia(\mathcal{S}^{(i)})}{\gamma_p -4\frac{(m-m_i)}{m_i}\max_{j\in\mathcal{P}^{(1)}}(\gamma_j)}<c,
	\end{equation}
for any $p\in\mathcal{P}^{(0)}$ and $i\in \{0,1\}$.
	Recall that $\gamma_i$ is defined as 
	\[
	\gamma_i=\exp({-r\|(DA)_i\|_2^2}).
	\]
	Then the separation property 
	\[
	\dist(\mathcal{S}^{(0)},\mathcal{S}^{(1)})>\max\{\dia(\mathcal{S}^{(0)}), \dia(\mathcal{S}^{(1)})\}
	\]
	implies that
	\[
	\lim\limits_{r\rightarrow+\infty}\max\limits_{i\in\mathcal{P}^{(1)}}(\gamma_i)=0 ~\mbox{and}~\lim\limits_{r\rightarrow+\infty}\frac{\min_{p'\in \mathcal{P}^{(1)}}\gamma_{p'}}{\max_{p\in \mathcal{P}^{(0)}}\gamma_{p}}=0.\]
If there exist $c,r$ such that
\begin{equation}
	\label{eqn:existc1_two}
	\frac{|\tau_q|}{\epsilon_i\dia(\mathcal{S}^{(i)})}>\frac{\max_{p'\in \mathcal{P}^{(1)}}|\frac{m}{2}(\rho-\gamma_{p'})|}{\gamma_p-\frac{4m_{1-i}}{m_i}\max_{i\in\mathcal{P}^{(1)}}(\gamma_i)}
	\end{equation}
for any $q\not\in\mathcal{Q}$ and $p\in \mathcal{P}^{(0)}$, then the inequalities \cref{eqn:ec_two} and \cref{eqn:condf4_two} hold true.
When taking $r\rightarrow+\infty$ on both sides of \cref{eqn:existc1_two}, the left hand side converges to  a finite constant and  the right hand side goes to  zero. 

By similar arguments, there exists a positive constant $r_0$ such that \cref{eqn:ec_two,eqn:cond_two102,eqn:condf3_two,eqn:condf4_two} hold true for any $r>r_0$.
\end{proof}

\subsection{Proof of \cref{theorem1}}
\label{sec:prf of three-cluster}
 The sufficient condition for exact clustering is similar as $2$-cluster case.
\begin{theorem}[sufficient condition for $k$-cluster]
\label{thm:suf_three}
 Assume $c,r\in \mathbb{R}^+$ and $\widehat{\Lambda}^T\in \mathcal{D}^\perp$.   $\vec{\textbf{e}}^{s,l}(0\leq s<l\leq k-1)$ are nonzero column vectors in $\mathbb{R}^n$. $\mathcal{P}^{(0)}$ and $\mathcal{P}^{(s,l)}$ are defined in \cref{eqn:P0} and \cref{eqn:Pkl}. Suppose for $p\in\mathcal{P}^{(s,l)}(0\leq s<l\leq k-1)$, $\widehat\Lambda_{pq}$ satisfies	
	\begin{subequations}
	\begin{align}
	&c\gamma_{p}\leq \frac{2}{m}|B_{pq}+\frac{m}{2}\widehat{\Lambda}_{pq}|,\label{eqn:cond10}\\
	&(B_{pq}+\frac{m}{2}\widehat{\Lambda}_{pq})-\frac{mc \gamma_{p}}{2}\sign(B_{pq}+\frac{m}{2}\widehat{\Lambda}_{pq})=\vec{\textbf{e}}^{s,l}_q,\label{eqn:cond11}
	\end{align}
	\end{subequations}
		 and for $p\in\mathcal{P}^{(0)}$, we have
		\begin{equation}
		\label{eqn:cond2}
		c\gamma_p>\frac{2}{m}|B_{pq}+\frac{m}{2}\widehat{\Lambda}_{pq}|.
		\end{equation}
		Besides, $\vec{\textbf{e}}$ meet the condition
		\begin{equation}
	\label{eqn:cond3}
		\vec{\textbf{e}}^{s,l}_q+\vec{\textbf{e}}^{l,m}_q=\vec{\textbf{e}}^{s,m}_q
		\end{equation}
		for all $0\leq s<l<m\leq k-1$.
		
Then $\widehat{\Lambda}$ is the optimal dual variable in \cref{eqn:dual}, and $\widehat{Y}$ can be derived respectively. Together with the fact that $\widehat{Y}=D\widehat{X}$, $\hat{X}$ satisfies the exact clustering property.
\end{theorem}

The condition \cref{eqn:cond3} is not seen in $2$-cluster case, which leads to different primal dual $(\widehat{Y},\widehat{\Lambda})$ construction for $k\geq 3$. 
For simplicity of notation, we only consider $k=3$ in \cref{theorem1} and the general case can be derived by the same method.

When $k=3$, under the condition in \cref{thm:suf_three}, we have 
\[\vec{\textbf{e}}^{0,1}+\vec{\textbf{e}}^{1,2}=\vec{\textbf{e}}^{0,2},\] and  $\widehat{X}$ has exact clustering property with
	\[\widehat{X}_{i}=
	\left\{ \begin{array}{l}
	\frac{m_1}{m}\vec{\textbf{e}}^{0,1}+\frac{m_2}{m}\vec{\textbf{e}}^{0,2}, \quad i\in\mathcal{I}^{(0)},\\
	\frac{m_1}{m}\vec{\textbf{e}}^{0,1}+\frac{m_2}{m}\vec{\textbf{e}}^{0,2}-\vec{\textbf{e}}^{0,1}, \quad i\in\mathcal{I}^{(1)},\\
	\frac{m_1}{m}\vec{\textbf{e}}^{0,1}+\frac{m_2}{m}\vec{\textbf{e}}^{0,2}-\vec{\textbf{e}}^{0,2}, \quad i\in\mathcal{I}^{(2)}.
	\end{array} \right.
	\]
Here we define
\[
	\left\{
	\begin{array}{ll}
	\epsilon_i=\frac{8(m-m_i)(m_i-1)+4m_i^2}{mm_i^2}\\
	\mu_{pq}=B_{pq}+\frac{m}{2}\widehat{\Lambda}_{pq}\\
	\tau^{s,l}_q=\frac{1}{m_km_l}\sum_{i\in\mathcal{P}^{(s,l)}}B_{iq}\\
	\rho^{s,l}=\frac{1}{m_km_l}\sum_{i\in\mathcal{P}^{(s,l)}}\gamma_{i}
	\end{array}
	\right.
	\]
	where $1\leq p \leq {m\choose 2}, 1\leq q\leq n, 0\leq i\leq 2, 0\leq s<l\leq 2
	$.
	
	For $p\in \mathcal{P}^{(1)}$ and $1\leq q\leq n$, construct specific $\widehat\Lambda_{pq}$ as 
	\begin{equation}
	\label{eqn:mu_pq}
	\frac{m}{2}\widehat\Lambda_{pq}=
	\left\{ \begin{array}{l}
	-B_{pq}+\tau^{0,1}_q+\frac{mc}{2}(\gamma_{p}-\rho^{0,1})\sign(\tau_q^{0,1})+m_2H_q,
	\quad \mbox{if}\ p\in\mathcal{P}^{(0,1)},\\
	-B_{pq}+\tau^{1,2}_q+\frac{mc}{2}(\gamma_{p}-\rho^{1,2})\sign(\tau_q^{1,2})+m_0H_q,
	\quad \mbox{if}\ p\in\mathcal{P}^{(1,2)},\\
	-B_{pq}+\tau^{0,2}_q+\frac{mc}{2}(\gamma_{p}-\rho^{0,2})\sign(\tau_q^{0,2})-m_1H_q,
	\quad \mbox{if}\ p\in\mathcal{P}^{(0,2)},\\
	\end{array} \right.
	\end{equation}
	where
	\[
	H_q=\frac{c}{2}\left(\rho^{0,1}\sign(\tau_q^{0,1})+\rho^{1,2}\sign(\tau_q^{1,2})-\rho^{0,2}\sign(\tau_q^{0,2})\right).
	\]

\cref{lem:suf_three1,lem:sufficient_three2} show proper sufficient condition of $\widehat{\Lambda}_{pq}$ to meet the Conditions \cref{eqn:cond10,eqn:cond11,eqn:cond2,eqn:cond3} in \cref{thm:suf_three}.
\begin{lemma}
\label{lem:suf_three1}
For each pair of $(s,l)$, suppose there exist $c,r\in \mathbb{R}^+$ such that
\begin{equation}
	\label{eq:eqn141}
	\left|\tau_q^{s,l}\right|-\frac{3mc}{2}\max_{i\in \mathcal{P}^{(1)}}\gamma_i>\max_{p\in \mathcal{P}^{(s,l)}}\frac{mc}{2}\left|\rho^{s,l}-\gamma_p\right|
	\end{equation}
	and
	\begin{equation}
	\label{eq:eqn151}
	\left|\tau_q^{s,l}\right|-\frac{3mc}{2}\max_{i\in \mathcal{P}^{(1)}}\gamma_i>\frac{mc}{2}\rho^{s,l}
	\end{equation}
	For $p\in \mathcal{P}^{(1)}$, $\widehat\Lambda_{pq}$ is defined in \cref{eqn:mu_pq}. Then $\widehat\Lambda$ meet the conditions \cref{eqn:cond10,eqn:cond11,eqn:cond3}. 
	\end{lemma}
	\begin{proof}
	By the definition of $H_q$, we have 
	\[
	|H_q|\leq \frac{3c}{2}\max_{i\in \mathcal{P}^{(1)}}\gamma_i.
	\]
	Combining with the condition \cref{eq:eqn141}, we get 
	\begin{equation}
	\label{eqn:inequ14}
	|\tau_q^{s,l}|>\max_{p\in \mathcal{P}^{(s,l)}}\frac{mc}{2}\left|\rho^{s,l}-\gamma_{p}\right|+m|H_q|.
	\end{equation}
	Then $\sign(\mu_{pq})=\sign(\tau^{s,l}_q)$ for $p\in \mathcal{P}^{(s,l)}$.

Secondly, provided that \cref{eq:eqn141,eq:eqn151} holding true, we obtain
\begin{equation}
	\label{conditionI1}
	\left|\tau^{0,1}_q+\frac{mc}{2}(\gamma_{p}-\rho^{0,1})\sign(\tau_q^{0,1})+m_2H_q\right|\geq \frac{m}{2}c\gamma_p,\quad p\in \mathcal{P}^{(0,1)},
	\end{equation}
	\begin{equation}
	\label{conditionI2}
	\left|\tau^{1,2}_q+\frac{mc}{2}(\gamma_{p}-\rho^{1,2})\sign(\tau_q^{1,2})+m_0H_q\right|\geq \frac{m}{2}c\gamma_p,\quad p\in \mathcal{P}^{(1,2)},
	\end{equation}
and\begin{equation}
	\label{conditionI3}
	\left|\tau^{0,2}_q+\frac{mc}{2}(\gamma_{p}-\rho^{0,2})\sign(\tau_q^{0,2})-m_1H_q\right|\geq \frac{m}{2}c\gamma_p, \quad p\in \mathcal{P}^{(0,2)}.
	\end{equation}
	Thus the inequality \eqref{eqn:cond10} holds true, and
	\[
	\left\{
	\begin{array}{ll}
	\vec{\mathbf{e}}_q^{0,1}=\tau^{0,1}_q-\frac{mc}{2}\rho^{0,1}\sign(\tau_q^{0,1})+m_2H_q,\\
	\vec{\mathbf{e}}_q^{1,2}=\tau^{1,2}_q-\frac{mc}{2}\rho^{1,2}\sign(\tau_q^{1,2})+m_0H_q,\\
	\vec{\mathbf{e}}_q^{0,2}=\tau^{0,2}_q-\frac{mc}{2}\rho^{0,2}\sign(\tau_q^{0,2})-m_1H_q,
	\end{array}
	\right.
	\]
	to get \cref{eqn:cond11,eqn:cond3}.
	\end{proof}
	
	\begin{lemma}
\label{lem:sufficient_three2}
Assume $\{\widehat{\Lambda}_{pq}\}_{p\in\mathcal{P}^{(s,l)},1\leq q\leq n}$ defined in \cref{eqn:mu_pq} satisfies \cref{eqn:cond10,eqn:cond11,eqn:cond3}. If
\begin{equation}
	\label{eqn:relaxc2}
	\begin{split}
	&\max_{0\leq i\leq 2}\left[{{\epsilon_i}\dia(\mathcal{S}^{(i)})}+\frac{4c (m-m_i)}{m_i}\max_{j\in\mathcal{P}^{(1)}}(\gamma_j)\right]<c\gamma_p,
	\end{split}
	\end{equation}
	for all $p\in \mathcal{P}^{(0)}$. Then there exist $c,r\in\mathbb{R}^+$ and $\widehat{\Lambda}\in\mathcal{D}^\perp$ satisfying Condition \cref{eqn:cond2}.
\end{lemma}
\begin{proof}
We show that, if $\{\widehat{\Lambda}_{pq}\}_{p\in\mathcal{P}^{(s,l)}}$ satisfies \cref{eqn:cond10,eqn:cond11,eqn:cond3}, then there exists $\{\widehat{\Lambda}_{pq}\}_{p\in\mathcal{P}^{(0)}}$,  such that $\widehat{\Lambda}$ belongs to $\mathcal{D}^\perp$, and
		\begin{equation}
		\label{eqn:compare}
		\begin{split}
		|\mu_{pq}|\leq& \max\limits_{0\leq i\leq 2}
		\left[\frac{m\epsilon_i\dia(\mathcal{S}^{(i)})}{2}+\frac{2c m(m-m_i)}{m_i}\max_{j\in\mathcal{P}^{(1)}}(\gamma_j)\right]
		\end{split}
		\end{equation}
		for any $ p\in \mathcal{P}^{(0)}$. Combined \cref{eqn:compare} with Condition \cref{eqn:relaxc2}, we can  get the conclusion. 

By direct calculation, ${\widehat{\Lambda}}^T\in \mathcal{D}^\perp$ is equivalent to,
		
		\begin{equation}\label{formula25}
		\left( \begin{array}{cc}
		-I_{(m_0-1)\times(m_0-1)},(D^{m_0-1})^T
		\end{array} \right){\Lambda}_{[{\mathcal{P}}_0^{(0)},q]}
		=-\left( \begin{array}{cccc}
		\underset{p\in\mathcal{P}^{(0,1)}_2\cup\mathcal{P}^{(0,2)}_2}\sum{\Lambda}_{pq},
		\dots,
		\underset{p\in\mathcal{P}^{(0,1)}_{m_0}\cup\mathcal{P}^{(0,2)}_{m_0}}\sum{\Lambda}_{pq}
		\end{array} \right)^T
		\end{equation}
		\begin{equation}
		\label{formula261}
		(D^{m_1})^T{\widehat{\Lambda}}_{[{\mathcal{P}}_1^{(0)},q]}\\
		=-\left(\begin{array}{cccc}
		\underset{p\in\mathcal{P}^{(1,2)}_{m_0+1}}\sum{\Lambda}_{pq},
		\underset{p\in\mathcal{P}^{(1,2)}_{m_0+2}}\sum{\Lambda}_{pq},
		\dots,
		\underset{p\in\mathcal{P}^{(1,2)}_{m_0+m_1}}\sum{\Lambda}_{pq}
		\end{array}\right)^T
		+\sum_{j=1}^{m_0}{\widehat{\Lambda}}_{[{\mathcal{P}}^{(0,1)}_j,q]}
		\end{equation}
		\begin{equation}
		\label{formula26}
		(D^{m_2})^T{\widehat{\Lambda}}_{[{\mathcal{P}}_2^{(0)},q]}=\sum_{j=1}^{m_0}{\widehat{\Lambda}}_{[{\mathcal{P}}^{(0,2)}_j,q]}+\sum_{j=m_0+1}^{m_0+m_1}{\widehat{\Lambda}}_{[{\mathcal{P}}^{(1,2)}_j,q]},
		\end{equation}
where \[
\mathcal{P}^{(0)}_k=\{(i-1)(m-1)+(j-i)\ |\ i<j,\  i,j\in\mathcal{I}^{(s)}\},
\]
and
\[{\mathcal{P}}^{(s,l)}_j=\{(j-1)(m-1)+(i-j)\ |\ i\in \mathcal{I}^{(l)}\}\] with $j\in \mathcal{I}^{(s)}$, and ${\widehat{\Lambda}}_{[{\mathcal{P}}_j^{(s,l)},q]}$ denotes the column vector whose elements are $\{\widehat{\Lambda}_{pq}\}_{p\in \mathcal{P}_j^{(s,l)}}$.

Consider the construction of $\widehat{\Lambda}_{[{\mathcal{P}}_2^{(0)},q]}$. Since $\{\widehat{\Lambda}_{pq}\}_{p\in\mathcal{P}^{(s,l)}}$ is defined in \cref{eqn:mu_pq}, we have the RHS of \eqref{formula26} equalling to zero by direct calculation.

Besides,
\begin{align}	
&\left\|\frac{m}{2}\left(\sum_{j=1}^{m_0}{\widehat{\Lambda}}_{[{\mathcal{P}}^{(0,2)}_j,q]}+\sum_{j=m_0+1}^{m_0+m_1}{\widehat{\Lambda}}_{[{\mathcal{P}}^{(1,2)}_j,q]}\right)\right\|_\infty\nonumber\\
		&\leq \frac{(m-m_2)(m_2-1)}{m_2}\dia(\mathcal{S}^{(2)})
		+cm(m-m_2)\max_{p\in \mathcal{P}^{(1)}}(\gamma_p).\nonumber
		\end{align}
Applying  \cite[Lemma 3]{zhu2014convex}, there exist $\widehat{\Lambda}_{pq} (p\in\mathcal{P}^{(0)}_2)$ satisfying
		\[
		\left|\frac{m}{2}{\widehat{\Lambda}}_{pq}\right|\leq
		\frac{2(m-m_2)(m_2-1)}{m_2^2}\dia(\mathcal{S}^{(2)})
		+\frac{2cm(m-m_2)}{m_2}\max_{i\in \mathcal{P}^{(1)}}(\gamma_i).
		\]
		Furthermore,
		\[
		\left|B_{pq}+\frac{m}{2}\widehat{\Lambda}_{pq}\right|\leq
		\frac{4(m-m_2)(m_2-1)+2m_2^2}{m_2^2}\dia(\mathcal{S}^{(2)})
		+\frac{2c m(m-m_2)}{m_0}\max_{i\in\mathcal{P}^{(1)}}(\gamma_i),\ p\in \mathcal{P}_2^{(0)},
		\]
		By the same analysis, there exist $\widehat{\Lambda}_{pq}(p\in\mathcal{P}^{(0)}_k)$  satisfying \eqref{formula25}, \eqref{formula261} and \eqref{formula26} with
		\[
		\left|B_{pq}+\frac{m}{2}\widehat{\Lambda}_{pq}\right|\leq
		\frac{4(m-m_k)(m_k-1)+2m_k^2}{m_k^2}\dia(\mathcal{S}^{(s)})
		+\frac{2c m(m-m_k)}{m_k}\max_{i\in\mathcal{P}^{(1)}}(\gamma_i),\ p\in\mathcal{P}_k^{(0)}.
		\]
		Therefore, the condition \cref{eqn:compare} is satisfied.
\end{proof}

Now, we begin to prove \cref{theorem1}.
\begin{proof}[Proof of \cref{theorem1}]
If there exist $c$ and $r$ such that they can guarantee the conditions \cref{eq:eqn141,eq:eqn151,eqn:relaxc2} in \cref{lem:suf_three1,lem:sufficient_three2}, we have optimal solution satisfying exact clustering property. 

Condition \cref{eqn:relaxc2} is equivalent to 
\begin{equation}
	\label{eqn:condf3}
	\gamma_p-\frac{4 (m-m_i)}{m_i}\max_{j\in\mathcal{P}^{(1)}}(\gamma_j)>0,
	\end{equation}
	and
	\begin{equation}
	\label{eqn:condf4}
	\frac{\epsilon_i\dia(\mathcal{S}^{(i)})}{\gamma_p -4 \frac{m-m_i}{m_i}\max_{j\in\mathcal{P}^{(1)}}(\gamma_j)}<c,
	\end{equation}
	for any $p\in\mathcal{P}^{(0)}$ and $0\leq i\leq 2$.

There exist $c$ and $r$ such that the condition \cref{eq:eqn141}  and \cref{eqn:condf4} hold true  if and only if
	\begin{equation}
	\label{eqn:existc1}
	\frac{|\tau_q^{s,l}|}{\epsilon_i\dia(\mathcal{S}^{(i)})}>\frac{\frac{m}{2}\left[\max_{p'\in \mathcal{P}^{(s,l)}}|(\rho^{s,l}-\gamma_{p'})|+3\max_{i\in \mathcal{P}^{(1)}}\gamma_i\right]}{\gamma_p-\frac{4(m-m_i)}{m_i}\max_{i\in\mathcal{P}^{(1)}}(\gamma_i)},
	\end{equation}
	for any $1\leq q\leq n$, $0\leq i\leq 2$, $0\leq s<l\leq 2$, and $p\in\mathcal{P}^{(0)}$.

	Since $
	\gamma_i=\exp({-r\|(DA)_i\|_2^2}),
	$
	the separation property \cref{eqn:sep} implies that
	$$\lim\limits_{r\rightarrow+\infty}\max\limits_{i\in\mathcal{P}^{(1)}}(\gamma_i)=0 ~\mbox{and}~\lim\limits_{r\rightarrow+\infty}\frac{\min_{p'\in \mathcal{P}^{(1)}}\gamma_{p'}}{\max_{p\in \mathcal{P}^{(0)}}\gamma_{p}}=0.$$
	Besides, by the assumption of \cref{theorem1}, the centres of each cluster are distinct in all dimensions. Therefore, when taking $r\rightarrow+\infty$ on both sides of \cref{eqn:existc1}, the left side converges to  a finite positive constant and  the right hand side goes to  zero.
	Similarly, the inequalities \cref{eq:eqn141,eq:eqn151,eqn:condf3,eqn:condf4} can also hold true when $r\rightarrow +\infty$.  Therefore, there exists a positive constant $r_0$ such that \cref{eq:eqn141,eq:eqn151,eqn:condf3,eqn:condf4} hold for any $r>r_0$.
\end{proof}

\section{The proof of \cref{thm:GaussianMixture}}
In order to prove \cref{thm:GaussianMixture}, we introduce some technical lemmas.
\begin{lemma}[\cite{LM2000}]
\label{main lemma}Let $(X_{1},...,X_{m})$ be i.i.d. Gaussian variables,
with mean $0$ and variance $1$. Let $\lambda=[\lambda_{1},...,\lambda_{m}]$ be a non-negative vector. Denote $Y=\sum_{i=1}^{m}\lambda_{i}X_{i}^{2}$. Then  for any $t>0$, we have
\[
\mathbb{P}(Y\geq\|\lambda\|_{1}+2\|\lambda\|_{2}\sqrt{t}+2\|\lambda\|_{\infty}t)\leq\exp(-t),
\]
\[
\mathbb{P}(Y\leq\|\lambda\|_{1}-2\|\lambda\|_{2}\sqrt{t})\leq\exp(-t).
\]
\end{lemma}
\begin{lemma}
\label{prop: Main proposition}Let $x$ sampled from $\mathcal{N}(\mu,\Sigma)$,
where $\Sigma\in\mathbb{R}^{n\times n}$ is the covariance matrix.
If $\mu=0$, then
\begin{equation}
\mathbb{P}(\|x\|_{2}^{2}\geq\mbox{\mbox{Tr}}(\Sigma)+2\|\Sigma\|_{F}\sqrt{t}+2\|\Sigma\|t)\leq\exp(-t),\qquad\mbox{for}\ t>0.\label{eq: upper bound inequality}
\end{equation}
 For arbitrary $\mu$, we get
\begin{equation}
\mathbb{P}\left(\|x\|_{2}^{2}\leq\|\mu\|_{2}^{2}+\mbox{Tr(\ensuremath{\Sigma})}-(2\|\Sigma\|_{F}+2\|\Sigma\|^{\frac{1}{2}}\|\mu\|_{2})\sqrt{t}\right)\leq2\exp(-t)\qquad\mbox{for}\ t>0.\label{eq: lower bound inequality}
\end{equation}
\end{lemma}
\begin{proof}
The singular value decomposition of $\Sigma$ is $U\varLambda U^{*}$,
where $U\in\mathbb{R}^{n\times n}$ is an unitary matrix and $\Lambda$
is a positive semi-definite diagonal matrix. Denote $\Lambda^{\frac{1}{2}}$
as the squared root of $\Lambda$. If $y\sim\mathcal{N}(0,I)$, 
$\|\Lambda^{\frac{1}{2}}y\|_{2}$ and $\|x\|_{2}$ have the same distribution
when $\mu=0$. Then applying Lemma \ref{main lemma} with $\lambda_{i}=\Lambda_{i}$,
where $\Lambda_{i}$ is the $i$th element in the diagonal of $\Lambda$, we can get (\ref{eq: upper bound inequality}). 

Now we prove inequality (\ref{eq: lower bound inequality}) when $\mu\neq0$. Since $U^{*}$ is unitary,
$\|x\|_{2}^{2}$ has the same distribution with $\|U^{*}x\|_{2}^{2}$. Then $U^{*}x\sim\mathcal{N}(\tilde{\mu},\Lambda)$,
where $\tilde{\mu}=U^{*}\mu$. Denote $z=\Lambda^{\frac{1}{2}}y+\tilde{\mu}$,
where $y\sim\mathcal{N}(0,I)$. Then $z$ has the same distribution
as $U^{*}x$. We can see that $\|z\|_{2}^{2}=\|\Lambda^{\frac{1}{2}}y\|_{2}^{2}+2\langle\tilde{\mu},\Lambda^{\frac{1}{2}}y\rangle+\|\tilde{\mu}\|_{2}^{2}$. By \cref{main lemma}, we obtain that 
\begin{equation}
\label{eqn:inequ1_appendix}
\mathbb{P}\left(\|\Lambda^{\frac{1}{2}}y\|_{2}^{2}\leq \mbox{Tr}(\Lambda)-2\|\Lambda\|_{F}\sqrt{t}=\mbox{Tr}(\Sigma)-2\|\Sigma\|_{F}\sqrt{t}\right)\leq\exp(-t).
\end{equation}

Meanwhile, we get that $\langle \tilde{\mu},\Lambda^{\frac{1}{2}}y\rangle$ is
distributed as $\mathcal{N}(0,\mu^{T}\Sigma\mu)$. Since
\[
\mathbb{P}\left(\left|\frac{\langle\tilde{\mu},\Lambda^{\frac{1}{2}}y\rangle}{\sqrt{\mu^{T}\Sigma\mu}}\right|>t\right)\leq\exp(-t^{2}/2),\quad \mbox{for}\ t>0
\]
by the tail bound of standard Gaussian distribution, therefore 
\begin{equation}
\label{eqn:inequ2_appendix}
\mathbb{P}\left(\langle\tilde{\mu},\Lambda^{\frac{1}{2}}y\rangle<-2\|\Sigma\|^{\frac{1}{2}}\|\mu\|_{2}\sqrt{t}\right)\leq\mathbb{P}\left(\frac{\langle\tilde{\mu},\Lambda^{\frac{1}{2}}y\rangle}{\sqrt{\mu^{T}\Sigma\mu}}<-2\sqrt{t}\right)\leq\exp(-t),
\end{equation}
for all $t>0$. For arbitrary random variables $X$ and $Y$, we have
\[
\mathbb{P}(X+Y\leq a+b)\leq\mathbb{P}(X\leq a)+\mathbb{P}(Y\leq b).
\]
Combining \cref{eqn:inequ1_appendix,eqn:inequ2_appendix}, we obtain
\[
\mathbb{P}(\|x\|_{2}^{2}\leq\|\mu\|_{2}^{2}+\mbox{Tr(\ensuremath{\Sigma})}-(2\|\Sigma\|_{F}+2\|\Sigma\|^{\frac{1}{2}}\|\mu\|_{2})\sqrt{t})\leq2\exp(-t)\qquad\mbox{for}\ t>0.
\]
\end{proof}
Now we begin to prove the theorem.
\begin{proof}[Proof of \cref{thm:GaussianMixture}]
Since the centers $\mu_k$ of gaussian distributions are distinct in each dimension, the mean values of samples inside each cluster are different in each dimension with probability $1$.
In order to have inner-class distance smaller than the inter-class
distance with high probability, we should have 
$
\|A_{i}-A_{j}\|_2>\|A_{k_1}-A_{k_2}\|_2
$
 for all $i\in\mathcal{S}^{(s)}$, $j\in\mathcal{S}^{(l)}$ ($0\leq s<l\leq k-1$) and $k_{1},k_{2}\in\mathcal{S}^{(i)}$($i=0,\cdots,k-1$).
 
 If $i\in\mathcal{S}^{(s)}$, $j\in\mathcal{S}^{(l)}$, we can see that $A_{i}-A_{j}\sim\mathcal{N}(\mu_{s}-\mu_{l},\Sigma_{s}+\Sigma_{l})$. Based on \cref{prop: Main proposition}, for fixed $s$ and $l$, we have
\begin{eqnarray*}
\mathbb{P}\biggl(\exists i\in\mathcal{S}^{(s)},j\in\mathcal{S}^{(l)}\quad  &  & \|A_{i}-A_{j}\|_{2}^{2}\leq\|\mu_{s}-\mu_{l}\|_{2}^{2}+\mbox{Tr}(\Sigma_{s}+\Sigma_{l})\\
 &  & -(2\|\Sigma_{s}+\Sigma_{l}\|_{F}+2\|\Sigma_{s}+\Sigma_{l}\|^{\frac{1}{2}}\|\mu_{s}-\mu_{l}\|_{2})\sqrt{3\log m}\biggr)\\
 &  & \leq \frac{2 m_km_l}{m^{3}}.
\end{eqnarray*}
Meanwhile, for $k_{1},k_{2}\in\mathcal{S}^{(i)}$($0\leq i\leq k-1$) with $k_1\neq k_2$, we get $A_{k_{1}}-A_{k_{2}}\sim\mathcal{N}(0,2\Sigma_{i})$ and the union bound for all clusters becomes
\begin{eqnarray*}
\mathbb{P}(\exists k_{1},k_{2}\in\mathcal{S}^{(i)}\quad & \|A_{k_{1}}-A_{k_{2}}\|_{2} & \geq\max_{i}(\mbox{Tr}(\Sigma_{i})+2\|\Sigma_{i}\|_{F}\sqrt{3\log m}+6\|\Sigma_{i}\|\log m))\\
 &  & \leq \frac{\sum_i m_i(m_i-1)}{2m^3}.
\end{eqnarray*}

Based on  the above two inequalities, if for all $0\leq s<l\leq k-1$, 
\begin{equation}
\begin{split}
\|\mu_{s}-\mu_{l}\|_{2}^{2}+\mbox{Tr}(\Sigma_{s,l}) & -  (2\|\Sigma_{s,l}\|_{F}+2\|\Sigma_{s,l}\|^{\frac{1}{2}}\|\mu_{s}-\mu_{l}\|_{2})\sqrt{3\log m}\\
 &  >\mbox{\ensuremath{\max}}_{i}(\mbox{\mbox{Tr}}(\Sigma_{i})+2\|\Sigma_{i}\|_{F}\sqrt{3\log m}+6\|\Sigma_{i}\|\log m),
 \end{split}
\end{equation}
where $\Sigma_{s,l}=\Sigma_{s}+\Sigma_{l}$,  we have $
\|A_{i}-A_{j}\|_2>\|A_{k_1}-A_{k_2}\|_2
$
 for all $i\in\mathcal{S}^{(s)}$, $j\in\mathcal{S}^{(l)}$ ($0\leq s<l\leq k-1$) and $k_{1},k_{2}\in\mathcal{S}^{(i)}$($i=0,\cdots,k-1$) with probability larger than $1-3/m$.
 
 Therefore, if 
 \begin{eqnarray*}
\|\mu_{s}-\mu_{l}\|_2 & > & \|\Sigma_{s,l}\|^{\frac{1}{2}}\sqrt{12\log m}+{(12\log m)}^{1/4}\|\Sigma_{s,l}\|_{F}^{1/2}\\
 &  & +\max_i\sqrt{{\mbox{Tr}}(\Sigma_{i})+2\|\Sigma_{i}\|_{F}\sqrt{3\log n}+2\|\Sigma_{i}\|3\log m},
\end{eqnarray*}
for all $0\leq s<l\leq k-1$, we have the exact recovery of model \cref{eqn:weightconvex}.
\end{proof}

\end{document}